\documentclass[11pt, reqno]{amsart}

\usepackage[top=3.75cm, bottom=3cm, left=3cm, right=3cm]{geometry}

\usepackage{url}
\usepackage{textcomp}
\usepackage{amsmath}
\usepackage{amsthm}
\usepackage{amssymb}
\usepackage{stmaryrd}
\usepackage[french, british]{babel}
\usepackage[utf8]{inputenc} 
\usepackage[T1]{fontenc}
\usepackage{tikz}
\usepackage{tikz-cd}     
\usepackage{centernot}
\usepackage{mathrsfs}
\usepackage[all]{xy}
\usepackage{calligra}
\usepackage{nicefrac}
\usepackage{mathtools}
\usepackage{hyperref}   
\usepackage[capitalise]{cleveref}	
\hypersetup{
    colorlinks=true,
    linkcolor=blue,
    filecolor=magenta,      
    urlcolor=cyan,
}
\usepackage{todonotes}    
\usepackage{enumitem}
\usepackage{graphicx}
\graphicspath{ {images/} }

\usepackage{multirow}

\theoremstyle{plain}
\newtheorem{theorem}{Theorem}[section]
\newtheorem*{theorem*}{Theorem}
\newtheorem*{theoreme*}{Théorème}
\newtheorem{question}{Question}
\newtheorem*{question*}{Question}

\newtheorem{lemma}[theorem]{Lemma}
\newtheorem{corollary}[theorem]{Corollary}
\newtheorem{proposition}[theorem]{Proposition}

\theoremstyle{definition}
\newtheorem{remark}[theorem]{Remark}

\theoremstyle{definition}

\newtheorem{example}[theorem]{Example}

\DeclareMathOperator{\Aut}{Aut}
\DeclareMathOperator{\Fix}{Fix}
\DeclareMathOperator{\Span}{Span}

\DeclareMathOperator{\im}{Im}
\DeclareMathOperator{\Id}{Id}

\DeclareMathOperator{\Spec}{Spec}

\DeclareMathOperator{\Supp}{Supp}
\DeclareMathOperator{\Pic}{Pic}
\DeclareMathOperator{\CPic}{\overline{Pic}}
\DeclareMathOperator{\Br}{Br}

\DeclareMathOperator{\Bl}{Bl}
\DeclareMathOperator{\Prym}{Prym}
\DeclareMathOperator{\Pres}{Pres}
\DeclareMathOperator{\SheafHom}{\mathcal{H\kern -1pt}\textit{om}} 
\DeclareMathOperator{\sm}{sm}

\DeclareMathOperator{\Ann}{Ann}
\DeclareMathOperator{\et}{\acute{e}t} 

\DeclareMathOperator{\CH}{CH}

\DeclareMathOperator{\D}{D}

\DeclareMathOperator{\Ku}{Ku}

\def\dar[#1]{\ar@<2pt>[#1]\ar@<-2pt>[#1]}

\DeclareMathOperator{\Coh}{\mathbf{Coh}}

\newcommand\matA{\mathbb{A}}

\newcommand\matC{\mathbb{C}}

\newcommand\matG{\mathbb{G}}

\newcommand\matP{\mathbb{P}}

\newcommand\matZ{\mathbb{Z}}

\newcommand\calA{\mathcal{A}}

\newcommand\calC{\mathcal{C}}
\newcommand\calD{\mathcal{D}}
\newcommand\calE{\mathcal{E}}
\newcommand\calF{\mathcal{F}}

\newcommand\calH{\mathcal{H}}
\newcommand\calI{\mathcal{I}}
\newcommand\calJ{\mathcal{J}}

\newcommand\calL{\mathcal{L}}

\newcommand\calO{\mathcal{O}}

\newcommand\calS{\mathcal{S}}

\newcommand\calU{\mathcal{U}}
\newcommand\calV{\mathcal{V}}
\newcommand\calW{\mathcal{W}}
\newcommand\calX{\mathcal{X}}
\newcommand\calY{\mathcal{Y}}

\newcommand{\fonction}[5]{\begin{array}{lrcl} 
#1: & #2 & \longrightarrow & #3 \\
    & #4 & \longmapsto & #5 \end{array}}

\newcommand{\mysetminusD}{\hbox{\tikz{\draw[line width=0.6pt,line cap=round] (3pt,0) -- (0,6pt);}}}
\newcommand{\mysetminusT}{\mysetminusD}
\newcommand{\mysetminusS}{\hbox{\tikz{\draw[line width=0.45pt,line cap=round] (2pt,0) -- (0,4pt);}}}
\newcommand{\mysetminusSS}{\hbox{\tikz{\draw[line width=0.4pt,line cap=round] (1.5pt,0) -- (0,3pt);}}}
\newcommand{\mysetminus}{\mathbin{\mathchoice{\mysetminusD}{\mysetminusT}{\mysetminusS}{\mysetminusSS}}}

\title[Intermediate Jacobian fibration and Pryms]{The Intermediate Jacobian fibration of a cubic fourfold containing a plane and fibrations in Prym varieties}
\date{}
\author{Dominique Mattei}

\address{Universit\"at Bonn, Endenicher Allee 60, 53115 Bonn, Germany.}
\email{dmattei@math.uni-bonn.de}

\begin{document}

\maketitle

\begin{abstract}
	We give a description of the intermediate Jacobian fibration attached to a general complex cubic fourfold $X$ containing a plane as a Lagrangian subfibration of a moduli space of torsion sheaves on the K3 surface associated to $X$ up to a cover. To do so, we propose a general construction of Lagrangian fibrations in Prym varieties as subfibrations of Beauville--Mukai systems over some loci of nodal curves in linear systems on K3 surfaces.
\end{abstract}

\tableofcontents

\section{Introduction}

A \textit{hyperk\"ahler manifold} $M$, also called \textit{irreducible symplectic manifold}, is a simply connected compact K\"ahler manifold admitting a holomorphic symplectic form $\sigma$ that generates $H^0(M,\Omega^2_M)$. Matsushita \cite{MatsushitaOnFibreSpaceStructHKVariety} proved that any non-constant morphism $\pi\colon M \to B$ from such a manifold $M$ to a normal smaller dimensional variety $B$ is a \textit{Lagrangian fibration}, i.e. the smooth fibres of $\pi$ are Lagrangian with respect to $\sigma$. In particular $\dim M=2\dim B$, the smooth fibres are abelian varieties, and when $B$ is smooth it is isomorphic to projective space \cite{HwangBaseManFibrationHKvariety}. 

A typical example of such a fibration is the so-called \textit{Beauville--Mukai systems} (see \S \ref{SectionCompactifiedJacobianLagrangianFibrations}). It consists of a moduli space of stable torsion sheaves on a K3 surface supported on curves in a linear system $B$. When smooth and projective, this space is a hyperk\"ahler manifold deformation equivalent to a Hilbert scheme of points on a K3 surface \cite{YoshiokaStabFMII}. The map associating to a sheaf its support leads to a Lagrangian fibration over $B$, whose fibre over a smooth curve $C\in B$ is isomorphic to the Jacobian $J(C)$.

Hyperk\"ahler manifolds share suprising links with Fano manifolds, i.e. manifolds with ample anticanonical bundle. For instance, Beauville and Donagi \cite{BeauvilleDonagiVarieteDroitesCubicFourfold} proved that the variety of lines $F(X)$ of a smooth cubic fourfold $X\subset \matP^5$ is a hyperk\"ahler manifold of dimension four. The Fano variety $F(Y)$ of a smooth hyperplane section $Y\subset X$ is Lagrangian in $F(X)$, and its Picard variety is isomorphic to the \textit{intermediate Jacobian} $J(Y)$ of $Y$ \cite{ClemensGriffithsIntermediateJacobianCubicThreefold}. Donagi and Markman \cite{DonagiMarkmanSpectralCovers} endowed the relative intermediate Jacobian $\pi\colon\calJ\to \calU$ over the open locus $\calU\subset (\matP^5)^\vee$ of smooth hyperplane sections of $X$ with a symplectic form. Twenty years later, based on the work of several authors, Laza-Saccà-Voisin \cite{LazaSaccaVoisinHKCompactificationIntermediateJacobianFibrationCubic4fold} and Saccà \cite{SaccaBiratGeomIntJacFibCubic4f} constructed a hyperk\"ahler compactification of $\calJ$ of OG$10$-type, extending $\pi$ to a Lagrangian fibration over $(\matP^5)^\vee$. 

In this paper, we consider the case of a general cubic fourfold $X$ containing a plane. Our first result is a description of $\calJ$ as a finite cover of a Lagrangian subfibration of a Beauville--Mukai system. 

\begin{theorem*}[= Theorem \ref{ThmMainPrymSubFibration}]
	There exists a degree $2$ polarized K3 surface $(S,H_S)$ and a commutative diagram
	\begin{center}
		\begin{tikzcd}
			\calJ \ar[r,"f"] \ar[d] & M\coloneqq M_S(0,5H_S,-10) \ar[d] \\
			\calU \ar[r,"i"] & \matP^{26}
		\end{tikzcd}
	\end{center}
	such that
	\begin{enumerate}[label=(\roman*)]
		\item the vertical arrows are Lagrangian fibrations,
		\item the horizontal arrows are quasi-finite,
		\item the image $f(\calJ)$ is smooth and symplectic in $M$. 
	\end{enumerate}
	
\end{theorem*}

The surface $S$ is the K3 surface associated to $X$ (in the sense of Hassett and Kuznetsov, see \cite{AddingtonThomasHodgeTheoryDerivedCatCubic4f}). It is a double cover of $\matP^2$ ramified over a sextic, and it is geometrically related to the variety of lines $F(X)$ of $X$ (see \S \ref{SectionGeomSettings}).

\begin{remark}
	The moduli space $M_S(0,5H_S,-10)$ is singular, however 
 $f(\calJ)$ lands in its smooth locus, see Lemma \ref{RmkPureDim1Rk1SheafIrredCurveIsStable}.
\end{remark}
\vspace{1\baselineskip}

The previous theorem is based on a general construction (see \S \ref{SectionCompactifiedJacobianLagrangianFibrations}) that we briefly explain now. Consider any K3 surface $S$ and a linear system $|\calL|$ such that the locally closed locus $V\subset |\calL|$ parametrizing curves with precisely $\delta>0$ nodes is not empty. Pick an Euler characteristic $\chi\in \matZ$ and consider the Beauville--Mukai system $\pi\colon M=M_S(0,[\calL],\chi)\to |\calL|$. For a suitable choice of $\chi$, the fibre of $\pi$ over a $\delta$-nodal curve $C\in V$ is the (shifted) \textit{compactified Jacobian} $\CPic^\delta(C)$. The latter is not smooth, but its singular locus admits an abelian part which is isomorphic to $\Pic^0(C^\nu)$ of the normalization $C^\nu \to C$.

Replace $V$ with one of its component and denote by $\calC \to V$ the universal curve over $V$. There exists a fibrewise resolution $\nu_V\colon \calC^\nu \to \calC$ over $V$, and the pushforward by $\nu_V$ gives a morphism
$$\Upsilon\colon \Pic^0_V(\calC^\nu) \hookrightarrow \CPic^\delta_V(\calC)$$
between the relative (compactified) Jacobians over $V$. 

\begin{theorem*}[=\cref{ThmSubfibrationJacobianNormalizations}]
	The image $N\coloneqq \Upsilon(\Pic^0_V(\calC^\nu))\subset M$ is a smooth symplectic subvariety of $M$ and $\pi|_N\colon N \to V$ is a Lagrangian subfibration of $\pi$.
\end{theorem*}
However, there is no evidence for the existence of a smooth compactification of $N$, as $V$ has no reason, a priori, to be rational.

Assume now that $p\colon S\to \matP^2$ is a double cover ramified over a smooth sextic $D_6$, and fix $\calL=\calO_S(d)\coloneqq p^*\calO_{\matP^2}(d)$ for some $d$.
For appropriate choices of $d$ and $\delta$, consider the locally closed subspace $\calV\subset |\calO_{\matP^2}(d)|$ parametrizing smooth plane curves of degree $d$ which intersect $D_6$ in $\delta$ points of multiplicity $2$. One can prove that the preimages by $p$ of these curves are $\delta$-nodal in $S$, so that $\calV$ is identified with a locally closed sublocus $\calV\simeq p^*\calV \subset V$ of positive codimension.
For $C\in \calV$, the map $p$ induces an étale double cover $C^\nu \to D$. One can associate to this cover the \textit{Prym variety} $\Prym(C^\nu/D)\subset \Pic^0(C^\nu)$, defined as the abelian subvariety consisting of line bundles fixed by $\tau\coloneqq \iota^*\circ (-)^\vee$, where $\iota\in\Aut(C^\nu)$ denotes the covering involution. The map $\tau$ can be defined in a relative way over $\calV$, leading to a relative Prym variety $\Prym_\calV(\calC^\nu/\calD)$, where $\calD$ denotes the induced family of curves in $\matP^2$.

On the other hand, there exists a symplectic involution $\tau_{M}$ on $M$ over $|\calO_S(d)|$, given by the composition of the cover involution of $S\to \matP^2$ and a relative (twisted) version of $(-)^\vee$. We prove that $\tau_M$ preserves $N$, and we obtain the following result.

\begin{theorem*}[=\cref{ThmPrymSubfibrationOverNodalCurves}]
		The map $\Upsilon\colon \Pic^0_V(\calC^\nu) \xrightarrow{\sim} N$ restricts to an isomorphism
		$$\Prym_{\calV}(\calC^\nu/\calD) \xrightarrow{\sim} P \coloneqq \Fix^0(N)_\calV,$$
		where $\Fix^0(N)$ is one connected component of the fixed locus of $\tau_M|_N$. In particular, $P$ is symplectic and smooth, and $\pi|_{P}\colon P \to \calV$ is a Lagrangian subfibration of $\pi$.
\end{theorem*}

A sum-up of these constructions are presented in \cref{RmkSumUpConstructionsFibrations}.
In many cases, the subspace $\calV\subset V$ identifies with an open $\calV\subset \matP^m$ in a projective space non-linearly embedded in $|\calL|$ (see for instance \cref{PropVIsOpenInP5VeroLinEmbed}). 

\subsection*{Structure of the paper}

In \S \ref{SectionCompJacNodalCurves}, we collect some results about the compactified Jacobian of nodal curves, and we describe its tangent space along its singular locus.
In \S \ref{SectionCompactifiedJacobianLagrangianFibrations}, we construct the varieties $N$ and $P$, and prove that they are symplectic in $M$. The main construction is stated in \cref{ThmPrymSubfibrationOverNodalCurves}.
In \S \ref{SectionGeomSettings}, we first introduce the geometric setup and some notations in order to properly state \cref{ThmMainPrymSubFibration}. Then, we prove some generality assumptions required throughout the text.
In \S \ref{sectionConstructionf}, we construct the map between the relative intermediate Jacobian $\calJ$ and the Beauville--Mukai system $M$, and we describe its image. We prove \cref{ThmMainPrymSubFibration} at the end of the section, using the construction of \S \ref{SectionSubfibPrymVar}.
In \S \ref{SectionQuestions}, we discuss some related works, questions and possible directions.

\subsection*{Acknowledgements}

I am deeply grateful to Daniel Huybrechts, who initiated this project, for his precious comments, advice and suggestions. I would like to thank Yajnaseni Dutta, Yoonjoo Kim, Andrés Rojas and Mauro Varesco for many helpful conversations, as well as the anonymous referee for many exposition suggestions and corrections. This research was funded by ERC Synergy Grant HyperK, Grant agreement ID 854361.

\section{Compactified Jacobians of nodal curves}\label{SectionCompJacNodalCurves}

\subsection{General results}

Let $f\colon X\to B$ be a locally projective, finitely presented, and flat morphism of schemes, whose geometric fibres are integral curves. We assume that all geometric fibres have the same arithmetic genus $p_a$. The (sheafified) relative Picard functor $(\Pic_{X/B})_{\et}$ is represented by a $B$-scheme (denoted $\Pic_B(X)$) which is smooth of relative dimension $p_a$. However, it is not projective over the base in general. Therefore, we define the \textit{compactified} Picard functor by allowing rank $1$ non-locally free torsion-free sheaves. More precisely, define the functor $\CPic_{X/B}$ from the category of $B$-schemes to the category of sets as:
\[ \CPic_{X/B}(T)\coloneqq\{\calI \text{ relative torsion-free rank } 1 \text{ sheaf on } X_T\} /{\Pic(T)}.
\]

Defining the \textit{degree} of a rank $1$ torsion-free sheaf $\calI$ as $\deg(\calI)=\chi(\calI)-\chi(\calO_{X_{\bar{s}}})$ (for a geometric fibre $X_{\bar{s}})$, we obtain subfunctors $\CPic^n_{X/B}$ parametrizing families of degree $d$ rank $1$ torsion-free sheaves, and (the scheme representing) $\CPic^0_{X/B}$ is called the \textit{relative compactified Jacobian}. The existence of the compactified Jacobian have been investigated by several authors (D'Souza, Esteves, Rego...). The general representability result is due to Altman and Kleiman, see \cite[Thm. 8.1, 8.5]{AltmanKleimanCompPicScheme}.

\begin{theorem}
	The functor $(\CPic_{X/B})_{\et}$ is representable by a $B$-scheme that we denote $\CPic_B(X)$, called the \textit{relative compactified Picard scheme}. Each subscheme $\CPic^d_B(X)$ is finitely presented and locally projective over $B$.
\end{theorem}

\begin{remark}\label{RmkSectionAllPicdAreIsomorphic}
	If $f\colon X\to B$ admits a section $\sigma\colon B\to X$, then all the $\Pic^d_B(X)$ (resp. $\CPic^d_B(X)$) are isomorphic to $\Pic^0_B(X)$ (resp. $\CPic^0_B(X)$), the isomorphism being given by the tensor product with $\calO_X(-d \cdot \sigma(B))$, which is a line bundle on $X$ of fibre degree $-d$. 
\end{remark}

Of course, $\Pic_B(X)=\CPic_B(X)$ when $f$ is smooth. For curves lying in surfaces, we can say more.

\begin{theorem}[\cite{AltmanIarrobinoKleimanIrredCompJac}]
	Suppose $X\to \Spec k$ is a proper integral curve of arithmetic genus $g$ over an algebraically closed field that lies on a smooth surface. Then the compactified Jacobian $\CPic^0_B(X)$ is an integral scheme of dimension $g$. It is both Cohen-Macaulay and a local complete intersection, and it contains $\Pic^0_B(X)$ as a dense open subscheme.
\end{theorem}

\subsection{Nodal curves}\label{SectionNodalCurves}

In this section, we gather some facts about the compactified Jacobian $\CPic^0(C)$ of a nodal curve $C$, following Kleppe's Ph.D. thesis \cite{KleppePhDThesisPicSchemeCurveComp}. Our goal is to describe the tangent space of $\CPic^0(C)$ at its singular points (see (\ref{EqnDecompTangentSpaceCompJac})).

Let $C$ be a nodal curve (by \cite[Corollary 9]{AltmanKleimanBertiniThms}, a curve with singularities of embedded dimension $\leq 2$ always lies in a smooth surface) of arithmetic genus $p_a$, with nodes $z_1,\dots,z_\delta\in C$, $\delta>0$. The normalization map $C^\nu \xrightarrow{\nu} C$ decomposes as a sequence
$$C^\nu=C_\delta \xrightarrow{\nu_\delta} C_{\delta-1} \to \cdots \xrightarrow{\nu_2} C_1 \xrightarrow{\nu_1} C_0=C$$
where each map is the partial desingularization of a single node. We let $\nu_i^{-1}(z_i)=\{x_i,y_i\}$.

\begin{theorem}[\cite{KleppePhDThesisPicSchemeCurveComp}, Theorem 3.2.1]
	Each $\nu_i\colon C_i \to C_{i-1}$, $i=1,\dots,\delta$, gives an exact sequence
	$$0 \to \matG_m \to \Pic^0(C_{i-1}) \xrightarrow{\nu_i^*} \Pic^0(C_i) \to 0.$$
\end{theorem}

Since extension of tori by tori are tori, in particular for any composition $\nu_{m,l}=\nu_m \circ \cdots \circ \nu_{l+1} \colon  C_m \to C_l$ (with $m>l$) we obtain the exact sequence
$$0 \to (\matG_m)^{m-l} \to \Pic^0(C_l) \to \Pic^0(C_m) \to 0.$$
More precisely, the data of a line bundle on $\Pic^0(C_l)$ is equivalent to the data of a line bundle $\calL\in \Pic^0(C_m)$ and scalars $\lambda_{1},\dots,\lambda_{m-l}$, where $\lambda_i$, $i=1,\dots,m-l$, identifies the fibres $\calL_{x_i}$ and $\calL_{y_i}$.

The compactification $\CPic^0(C)$ can be constructed (at least topologically) by induction from $\CPic^0(C^\nu)=\Pic^0(C^\nu)$ as follows. For each $i=1,\dots,\delta$, the $\matG_m$-bundle $\Pic^0(C_{i-1})\to \Pic^0(C_{i})$ admits a compactification as a $\matP^1$-bundle $\Pres_i \to \Pic^0(C_{i})$ by adding two sections $s_0$ and $s_\infty$.

\begin{remark}
    The construction of $\Pres_i$ is functorial, see \cite[Chap. II]{KleppePhDThesisPicSchemeCurveComp} for the definition of the \textit{presentation functor} $\Pres_{C_{i}/C_{i-1}}$.
\end{remark}

\begin{theorem}[\cite{KleppePhDThesisPicSchemeCurveComp}, Lemma 7.2.1]
	For any $i=1,\dots,\delta$, the underlying topological space of the compactified Jacobian $\CPic^0(C_{i-1})$ is obtained from $\Pres_{i} \to \Pic^0(C_i)$ by identifying the sections $s_0$ and $s_\infty$ via the translation by the line bundle $\calO_{C_{i}}(x_i-y_i)$\footnote{See \cite[Figure 1]{SawonSingularFibresVeryGeneralLagFib} for a nice picture.}.
\end{theorem}

The compactified Jacobian $\CPic^0(C)$ is not smooth, but its singular locus is exactly the non-locally free locus, i.e. $\CPic^0(C)_{sing}=\partial \CPic^0(C)\coloneqq\CPic^0(C)\mysetminus \Pic^0(C)$. Moreover, this singular locus admits a stratification in terms of the orbits of the action of $\Pic^0(C)$ on $\CPic^0(C)$ by tensor product.

\begin{proposition}\label{PropOrbitsDecompSingLocusCompJacobian}
	For each $i=1,\dots \delta$, there exists $\binom{\delta}{i}$ orbits of codimension $i$, each of which is isomorphic to $\Pic^0(C_i)$ (a $\Pic^0(C_i)$-torsor) where $C_i\to C$ is the partial desingularization of $i$ nodes, and $\partial \CPic^0(C)$ is the union of these orbits.
\end{proposition}

\begin{proof}
	Pick a degree $0$ rank $1$ torsion-free sheaf $F\in\partial\CPic^0(C)$. Up to reordering, assume that $F$ is not locally free only over the nodes $\{z_1,\dots,z_m\}$, and consider $\mu\colon  C_m\to C$ the partial desingularization of these nodes. Fix a base point $\calL_m\in \Pic^{-m}(C_m)$. One can check that $F$ is in fact of the form $\mu_* \calF_m$ for the degree $-m$ line bundle $\calF_m\coloneqq\mu^*F/tors \in \Pic^{-m}(C_m)$, and in particular $\calF_m\simeq \calL_m\otimes\calL_0$ for some $\calL_0\in \Pic^0(C_m)$. By the surjectivity of $\mu^*\colon \Pic^0(C)\to \Pic^0(C_m)$, there is some line bundle $\calL\in \Pic^0(C)$ such that $\calL_0=\mu^*\calL$. By the projection formula, we get $F\simeq \mu_*\calL_m \otimes \calL$, and therefore $F$ lies in the orbit $O$ of $\mu_*\calL_m$.
	
	Moreover, the action of $\Pic^0(C)$ factors to a free and transitive action of $\Pic^0(C_m)$ on the orbit of $\mu_*\calL_m$. Indeed, pick $\calL\in \Pic^0(C)$, we have
	\begin{eqnarray*}
		& & \calL \otimes \mu_*\calL_m \simeq \mu_*\calL_m \\
		&\iff & \calL \otimes \mu_*\calO_{C_m} \simeq \mu_*\calO_{C_m} \\
		&\iff & \mu_*\mu^*\calL \simeq \mu_*\calO_{C_m} \\
		&\iff & H^0(C_m,\mu^*\calL)\neq 0 \\
		&\iff & \mu^*\calL\simeq \calO_{C_m}
	\end{eqnarray*}
	since $\calL$ has degree $0$. The orbit $O$ is isomorphic to $\Pic^0(C_m)$ (because $\Pic^0(C_m)$ is smooth and $\Pic^0(C_m)\to O$ is faithfully flat, see \cite[\S \textit{Group actions}, Prop. $1.65$]{MilneAlgebraicGroups}), which has dimension $\dim \Pic^0(C)-i$.
	
	Finally, the number of orbits of a fixed codimension $i$ is given by the number of choices of a subset $\{ z_{j_1},\dots z_{j_i}\}\subset \{z_1,\dots,z_\delta\}$, which is $\binom{\delta}{i}$.

\end{proof}

More can be said about the local structure of the singular locus $\CPic^0(C)$. 

\begin{proposition}[\cite{KleppePhDThesisPicSchemeCurveComp}, Proposition 6.2.2]
	
	The completion of the local ring of $\CPic^0(C)$ at a point $[F]$ in an orbit of codimension $m$ is of the form
	\begin{eqnarray}\label{EqnFormalLocalRingCompPicSingularLocus}
	\widehat{\calO}_{\CPic^0(C),[F]}\simeq \widehat{\bigotimes}_{j=1}^{m}\dfrac{\matC[[ X_j,Y_j]]}{X_jY_j} \widehat{\bigotimes}\matC[[T_1,\dots,T_{p_a-m}]].
	\end{eqnarray}
\end{proposition}

For the rest of the paper, we will only need to consider tangent spaces (instead of formal neighborhoods) on the singular locus of $\CPic^0(C)$. 

%
%
%

\begin{lemma}\label{LemmaSmoothResolutionPdeltaFibration}
	There exists a smooth resolution $P^{\sm} \to \CPic^0(C)$ which is a $(\matP^1)^{\delta}$-fibration over $\Pic^0(C^\nu)$
\end{lemma}

\begin{proof}
	We proceed by induction. If $\delta=1$, then $\Pres_1 \to \CPic^0(C)$ is a smooth resolution and it is a $\matP^1$-bundle over $\Pic^0(C^\nu)$.
	
	Now assume that the lemma is true for $\delta-1\geq 1$. Then there is a smooth resolution $P_1^{\sm}$ of $\CPic^0(C_1)$ and a $(\matP^1)^{\delta-1}$-fibration $P_1^{\sm} \to \Pic^0(C^\nu)$. Then consider the fibre product
	$$\begin{tikzcd}
	P^{\sm} \ar[d] \arrow[dr, phantom, "\lrcorner", very near start] \ar[r] & \Pres_1 \ar[d,"\matP^1"] \\
	P_1^{\sm} \ar[r] & \CPic^0(C_1)
	\end{tikzcd}$$
	Since the base change of a $\matP^1$-bundle is a $\matP^1$-bundle \cite[\href{https://stacks.math.columbia.edu/tag/01O3}{Tag 01O3}]{stacks-project}, $P^{\sm}$ is a $\matP^1$-bundle over a $(\matP^1)^{\delta-1}$-fibration over $\Pic^0(C^\nu)$ (which is smooth), hence it is smooth and a $(\matP^1)^\delta$-fibration over $\Pic^0(C^\nu)$. Obviously $P^{\sm}$ is birational to $\Pres_1$, which is itself birational to $\CPic^0(C)$, therefore it is a smooth resolution of $\CPic^0(C)$.

\end{proof}

\begin{remark}
	The resolution $P^{\sm} \to \CPic^0(C)$ is in fact the normalization map. This is proved in \cite{OdaSeshadriCompGenJacVar}, but can also be deduced from (\ref{EqnFormalLocalRingCompPicSingularLocus}). When $\delta=1$, it is also proved (in more general settings) in \cite[Prop. 2.2]{HwangOguisoCharFoliationDiscHypHolLagFib}.
\end{remark}


Thanks to \cref{LemmaSmoothResolutionPdeltaFibration}, we can find a basis of the tangent space (in fact, of the tangent cone) of $\CPic^0(C)$ at a singular point.

Start with a sheaf $[F]\in \CPic^0(C)$ and assume that it is not locally free only over the nodes $z_1,\dots,z_m$. Recall that $\CPic^0(C)$ is the quotient of a $\matP^1$-bundle $\Pres_1$ over $\CPic^0(C_1)$, identifying two sections $s_0,s_\infty$ by the translation by $\calO_{C_1}(x_1-y_1)$. Then there is some $F_1\in\CPic^0(C_1)$ such that the preimage of $F$ in $\Pres_1$ is given by $\{(F_1,0),(F_1(x_1-y_1),\infty)\}$. The image in $T_{\CPic^0(C),[F]}$ of the tangent spaces $T_{s_0,(F_1,0)}$ and $T_{s_\infty,(F_1(x_1-y_1),\infty)}$ coincide (and both are identified with $T_{\CPic^0(C_1),[F_1]}$), thus we get
$$T_{\CPic^0(C),[F]}\simeq T_{\CPic^0(C_1),[F_1]}\oplus \matC^2.$$
The generators of the summand $\matC^2$ can be chosen to be the images by (the differential of) $\Pres_1\to \CPic^0(C)$ of tangent vectors $t_0\in T_{\Pres_1,(F_1,0)}$ and $t_\infty\in T_{\Pres_1,(F_1(x_1-y_1),\infty)}$ which are tangent to the $\matP^1$-fibres (these images do not coincide, thanks to (\ref{EqnFormalLocalRingCompPicSingularLocus})). 

We can proceed from here by induction. Indeed, $\CPic^0(C_1)$ is the quotient of a $\matP^1$-bundle $\Pres_2$ over $\CPic^0(C_2)$ identifying two sections $s^2_0$ and $s_\infty^2$ up to translation by $\calO_{C_2}(x_2-y_2)$. As before, one can choose $F_2\in \CPic^0(C_2)$ and by the same argument we get
$$T_{\CPic^0(C_1),[F_1]}\simeq T_{\CPic^0(C_2),[F_2]}\oplus \matC^2,$$
where we identify again $T_{s_0^2,(F_2,0)}\simeq T_{s_\infty^2,(F_2(x_2-y_2)),\infty} \simeq T_{\CPic^0(C_2),[F_2]}$. The generators of the summand $\matC^2$ can be chosen to be the images of tangent vectors $t^2_0\in T_{\Pres_2,(F_2,0)}$ and $t^2_\infty\in T_{\Pres_2,(F_2(x_2-y_2),\infty)}$ tangent to the $\matP^1$-fibres.

Consider the fibre product
$$\begin{tikzcd}
\Pres_{1,2} \arrow[dr, phantom, "\lrcorner", very near start] \ar[d,"p"] \ar[r,"q"] & \Pres_1 \ar[d] \\
\Pres_2 \ar[r] & \CPic^0(C_1).
\end{tikzcd}$$
One can check that $t_0$ and $t_\infty$ are in the image of the differential map $\text{d}q$ at the preimages of $(F_1,0)$ and $(F_1(x_1-y_1),\infty)$ by $q$. In other words, we can choose $G_0, \ G_\infty\in q^{-1}(F_1,0)$ and $q^{-1}(F_1(x_1-y_1),\infty)$ respectively, and tangent vectors $t_0^1$ and $t_\infty^1$ tangent to the $\matP^1$ fibres of $p$ such that $\text{d}q(t_0^1)=t_0$ and $q(t_\infty^1)=t_\infty$. 

%

We can repeat this process $m$-times, after which the point $F_m\in\CPic^0(C_m)$ that we obtain is locally free (that is, smooth in $\CPic^0(C_m)$), so that $P^{\sm}\to \CPic^0(C_m)$ is a $(\matP^1)^{\delta-m}$-bundle on a neighborhood of $F_m$.

We conclude that in the preimage $\mathsf{sm}^{-1}[F]$ of $\mathsf{sm}\colon P^{\sm}\to \CPic^0(C)$, we can choose points $F_i$ and vectors $t_i\in T_{P^{\sm},[F_i]}$, $i=1,\dots 2m$, such that
\begin{eqnarray}\label{EqnDecompTangentSpaceCompJac}
T_{\CPic^0(C),[F]}\simeq T_{\Pic^0(C^\nu),[\calO_{C^\nu}]} \oplus \bigoplus_{1}^{2m} \matC t_i \oplus  \matC^{\delta-m},
\end{eqnarray}
where by abuse of notation we wrote $t_i$ for the images of $t_i$ by the differential $\text{d}\mathsf{sm}$.

\begin{remark}\label{RmkSummandCnuComesFromAnySection}
	Each $F_i$ belongs to some section $s_i$ of the $(\matP^1)^\delta$-bundle $P^{\sm}\to \Pic^0(C^\nu)$, hence the summand in (\ref{EqnDecompTangentSpaceCompJac}) isomorphic to $T_{\Pic^0(C^\nu),[\calO_{C^\nu}]}$ is the image of $T_{s_i,[F_i]}$ by $\text{d}\mathsf{sm}$ \textit{for all} $i$.
\end{remark}

\begin{remark}\label{RmkJdAreIsoAndTangentSpaceAreIso}
	As stated in \cref{RmkSectionAllPicdAreIsomorphic}, all the $\Pic^d_C$ (resp. $\CPic^d_C$) are isomorphic, for $d\in\matZ$, up to the choice of a base point in each $\Pic^d(C)$. Therefore, the previous discussion aditionally gives a description of the tangent space of any $\CPic^d(C)$. Moreover, since $\Pic^d(C)$ is a $\Pic^0(C)$-torsor, all its tangent spaces are naturally isomorphic to each other (the isomorphism $T_{\Pic^d(C),\calL_1} \xrightarrow{\sim} T_{\Pic^d(C),\calL_2}$ given by the differential $\text{d}m$ of the multiplication map $m\colon \calL\mapsto \calL\otimes \calL_1^\vee \otimes \calL_2$), and therefore the same statement holds for the tangent spaces within each orbits. For this reason, we do not pay too much attention to the choice of the basepoint of the tangent space.
\end{remark}

\section{Lagrangian fibrations over nodal curves on K3 surfaces}\label{SectionCompactifiedJacobianLagrangianFibrations}

Let $S$ be a K3 surface, and $C_0\subset S$ be a smooth curve of genus $p_a$. Denote $B\coloneqq |C_0|\simeq \matP^{p_a}$ the linear system on $S$ of curves linearly equivalent to $C_0$. Consider the $2p_a$-dimensional moduli space of semistable torsion sheaves $M\coloneqq M_H(0,[C_0],\chi\coloneqq 1-p_a+d)$ on $S$ with respect to a polarization $H\in\Pic(S)$ for some $d\in\matZ$. Choosing $H$ generic enough, we can assume that the singular locus of $M$ lies in the strictly semistable locus, and it is empty whenever the Mukai vector $(0,[C_0],\chi)$ is primitive. Moreover, Mukai showed in \cite{MukaiSympStructModuliSheavesAbK3Surf} that the stable locus of $M$ admits a holomorphic symplectic form $\sigma_M$.



Following \cite{BeauvilleSystemesHamiltoniens}, (see also \cite{ASFRelativePrymVar}), there is a \textit{support morphism}
$$\pi\colon M \to B$$
which sends a sheaf $F$ to its support $\pi(F)\coloneqq\Supp(F)$. The fibre of $\pi$ over an integral curve $C\in B$ is isomorphic to $\CPic^d(C)$. Set $\calU\subset B$ the open locus of integral curves, then we have an isomorphism
\begin{eqnarray}\label{EqnModuliK3IsRelCompJac}
M_\calU\coloneqq\pi^{-1}(\calU) \simeq \CPic^d_\calU(\calC)
\end{eqnarray}
where $\calC\to \calU$ denotes the universal family of curves associated to the linear system. 

\begin{remark}\label{RmkPureDim1Rk1SheafIrredCurveIsStable}
	A direct computation of Hilbert polynomials reveals that any pure sheaf of rank $1$ supported on an integral curve is automatically stable with respect to any polarization. In particular, $M_\calU$ is smooth.
\end{remark}

Beauville \cite{BeauvilleSystemesHamiltoniens} has shown that the map $\pi\colon M \to B$ is a \textit{Lagrangian fibration} (that is, the general fibre of $\pi$ is Lagrangian with respect to $\sigma_M$) called \textit{Beauville--Mukai integrable system}. We will use a slightly stronger result.

\begin{lemma}[\cite{MatsushitaSingularFibresLagFib}, Lemma 2.3]\label{LemmaPullbackSympFormOnResolutionFibreIsTrivial}
	For any curve $C\in \calU$, the smooth resolution $\mathsf{sm}\colon P^{\sm} \to \CPic^d(C)$ satisfies
	$$\mathsf{sm}^*\sigma_M|_{\CPic^d(C)}=0.$$
\end{lemma} 

For the previous lemma, we use that $M_\calU$ is smooth (\cref{RmkPureDim1Rk1SheafIrredCurveIsStable}), in particular $\pi_\calU$ is flat by miracle flatness \cite[Exer. III.10.9]{HartshorneAlgebraicGeometry}.
\newline



For integers $0\leq \delta \leq p_a$, one can define the \textit{Severi variety}
$$V_{B,\delta}\subset B$$
as the subvariety parametrizing irreducible curves having $\delta$ nodes as the only singularities. It is a quasi-projective variety, functorially well defined and, when non-empty, each of its irreducible components is smooth of expected dimension $p_a-\delta$ (see \cite{TannenbaumFamiliesCurvesNodesK3Surfaces}). Up to our knowledge, the most general result of non-emptyness is the following.

\begin{theorem}[\cite{ChenRationalCurvesK3Surfaces}]
	Let $(S,H)$ be a general primitively polarized K3 surface of genus at least $3$. Then for each $m\geq 1$, the Severi variety $V_{|mH|,\delta}$ is non-empty.
\end{theorem}
%

\subsection{Sub-fibration over nodal curves}\label{SectionSubFibNodalCurves}

We keep the notation of \S \ref{SectionCompactifiedJacobianLagrangianFibrations}. Assume that there exists a non-empty smooth irreducible component $V\subset V_{B,\delta}$ (which, in particular, lies in $\calU$) of dimension $\dim V=p_a-\delta$. For simplicity, assume $d=\delta$, see \cref{RmkSubfibJacobianAnyDegree}.

\begin{theorem}\label{ThmSubfibrationJacobianNormalizations}
	There exists a smooth subvariety $N\subset M$ with $\pi(N)=V$, such that for any $C\in V$, the fibre of $N\to V$ is given by $N_C\simeq \Pic^0(C^\nu)$ for $C^\nu$ the normalization of $C$. Moreover, the restricted symplectic form $\sigma_N\coloneqq\sigma_M|_{N}$ is still nondegenerate, hence $N$ is smooth and symplectic and $\pi_N\colon N \to V$ describes a Lagrangian subfibration of $\pi\colon M\to B$.
\end{theorem}

	Before the proof, we state two constructions that we use. Let $\calC\subset S\times V \to V$ be the restriction of the universal family of curves over $B$.
	
	\begin{lemma}\label{LemmaFibrewiseNormalizationCurves}
		The family $\calC$ admits a fibrewise resolution, i.e. there exists a family $\calC^\nu\to V$ of smooth curves and a map $\nu_V\colon \calC^\nu\to\calC$ over $V$ such that the fibre $\nu_C$ of $\nu_V$ over $[C]\in V$ identifies with the normalization $\nu_V\colon \calC^\nu_{[C]}\simeq C^\nu \to C$. Moreover, $\nu_V$ is the normalization map.
	\end{lemma} 
	
	\begin{proof}
	By a result of Teissier \cite[I, Thm. 1.3.2]{TeissierResolutionSimultanee} (see also \cite{DedieuSernesiEquigenericityEquisingularity}), the existence of a fibrewise resolution is equivalent to the family of curves being flat and having constant $\delta$-invariant (or, equivalently, having constant geometric genus). The latter holds in our case because all fibres of $\calC\to V$ are $\delta$-nodal. Since the universal line over $B$ is flat over $B$, the restriction $\calC\to V$ is flat.
	\end{proof}

	\begin{lemma}
		There exists a natural map $\Upsilon\colon \Pic^0_V(\calC^\nu) \to \CPic^\delta_V(\calC)$ over $V$, fibrewise given by the pushforward $\nu_*\colon \Pic^0(C^\nu) \to \CPic^\delta(C)$.
	\end{lemma}
	
	\begin{proof}
		On the level of functors, we define for any $T\to V$ the map
		$$\fonction{\Upsilon}{\Pic_{\calC^\nu/V}(T)}{\CPic_{\calC/V}(T)}{\calL\in\Coh(\calC^\nu_T)}{(\nu_T)_*\calL\in\Coh(\calC_T)}$$
		where $\nu_T\colon \calC^\nu_T \to \calC_T$ is the base change of $\nu_V\colon \calC^\nu \to \calC$ along $T\to V$.
		
		This map is well defined. Indeed, assume that $\calL$ is a coherent sheaf which is fibrewise locally free over $T$. Then $(\nu_T)_*\calL$ is of rank $1$ since $\nu_V$ is birational. Hence it suffices to prove that $(\nu_T)_*\calL$ is fibrewise torsion free. This is a consequence of the general following remark. Let $A,B$ be two rings $\varphi\colon B\to A$ a homomorphism, and $M$ a $A$-module. Then, for any $m\in M$, $\Ann_B(m)\subset \varphi^{-1}\Ann_A(m)$. 

	\end{proof}

	

	\begin{proof}[Proof of \cref{ThmSubfibrationJacobianNormalizations}]
	
	Recall from (\ref{EqnModuliK3IsRelCompJac}) that $\CPic_V^\delta(\calC)\simeq M_V$. We prove that $\Upsilon$ is a closed immersion. Since both $\Pic^0_V(\calC^\nu)$ and $M_V$ are flat over $V$ (recall that $\Pic^0_V(\calC^\nu)$ is smooth over $V$ \cite[Prop. 5.19]{KleimanPicardScheme}), by \cite[\href{https://stacks.math.columbia.edu/tag/039A}{Tag 039A}]{stacks-project} it is enough to prove that $\Upsilon\colon \Pic^0_V(\calC^\nu)\to M_V$ is fibrewise a closed immersion. Hence we just need to prove that $\nu_*\colon \Pic^0(C^\nu)\to \CPic^\delta(C)$ is a closed immersion for each $C\in V$. The latter holds because $\nu_*$ maps $\Pic^0(C^\nu)$ isomorphically onto the unique codimension $\delta$ orbit of the action of $\Pic^0(C)$ in $\CPic^\delta(C)$ (see the proof of \cref{PropOrbitsDecompSingLocusCompJacobian}).
	
	
	Now we must show that the restriction $\sigma_N$ of $\sigma_M$ on $N\coloneqq \Upsilon(\Pic^0_V(\calC^\nu))\subset M$ is non-degenerate. Fix a curve $[C]\in V$. By the action of $\Pic^0(C)$, it is enough to prove that $\sigma_N$ is non-degenerate on $T_{N,[\nu_*\calO_{C^\nu}]}$. Denote $F\coloneqq\nu_*\calO_{C^\nu}$. Since $M$ is smooth at $[F]$, the symplectic form $\sigma_M$ restricts to a symplectic form $\sigma_{[F]}$ on the tangent space $T_{M,[F]}$, which by (\ref{EqnDecompTangentSpaceCompJac}) decomposes as
	\begin{eqnarray} \label{eq:decomp.tangent.moduli}
	T_{M,[F]} \simeq T\oplus T_{\CPic^\delta(C),[F]} \simeq  T \oplus T_{\Pic^0(C^\nu),[\calO_{C^\nu}]}\oplus \bigoplus_{i=1}^{2\delta}\matC t_i,
	\end{eqnarray}
	where $T$ is a chosen lift of the subspace $\im(d\pi_{[F]})\subset T_{B,C}$, and the $t_i$'s are the images of vectors tangent to preimages $[F_i]$ of $[F]$ by a smooth resolution $\mathsf{sm}\colon P^{\sm}\to \CPic^d_C$. 
	\begin{lemma}\label{LemmaImageDiffIsTangV}
		We have $\im(d\pi_{[F]})=T_{V,[C]}$.
	\end{lemma}
	\begin{proof}
		We know that $\pi_N\colon N \to V$ is smooth, hence $d\pi_N$ is everywhere surjective, in particular $\im(d\pi_{N,[F]})=T_{V,[C]} \subset \im(d\pi_{[F]})$. But from (\ref{eq:decomp.tangent.moduli}) we have  $\dim T_{M,[F]}=2p_a=\dim T+p_a+2\delta$, i.e. $\dim \im(d\pi_{[F]})=\dim T=p_a-\delta=\dim T_{V,[C]}$ and the result follows.
	\end{proof}
 
	Pick any vector $u\in T_{\Pic^0(C^\nu),[\calO_{C^\nu}]}\subset T_{M,[F]}$. Then for any $i=1,\dots,2\delta$, both $t_i$ and $u$ are of the form $t_i=\text{d}\mathsf{sm}_{[F_i]}t_i'$, $u=\text{d}\mathsf{sm}_{[F_i]}u_i'$ for some $t_i',u_i'\in T_{P^{\sm},[F_i]}$ (\cref{RmkSummandCnuComesFromAnySection}). Since $\mathsf{sm}^*(\sigma_M)=0$ by \cref{LemmaPullbackSympFormOnResolutionFibreIsTrivial} we get
	\begin{eqnarray*}
	\sigma_M(t_i,u)&=& \sigma_M(\text{d}\mathsf{sm}_{[F_i]}t_i',\text{d}\mathsf{sm}_{[F_i]}u_i') \\
	&=&\mathsf{sm}^*\sigma_M(t_i',u_i') \\
	&=& 0.
	\end{eqnarray*}
	In other words, we get $T_{\CPic^\delta(C),[F]}\subseteq T_{\Pic^0(C^\nu),[\calO_{C^\nu}]}^\perp$ and by a dimension count, equality must hold.
 Therefore the sympletic form $\sigma_M$ induces an isomorphism
	$$T_{\Pic^0(C^\nu),[\calO_{C^\nu}]} \xrightarrow{\sim} T_{M,[F]}/T_{\CPic^\delta(C),[F]} \simeq T.$$
Therefore $\sigma_M$ restrict to a sympletic form on $N$ and the restricted map $\pi_N\colon N\to V$ is a smooth Lagrangian fibration.
	
\end{proof}

%

\begin{remark}\label{RmkSubfibJacobianAnyDegree}
	With the choice $d=\delta$, $N\simeq \Pic^0_V(\calC^\nu)$ is an abelian scheme over $V$, with a $0$-section given over $C\in V$ by $\calO_{C^\nu}$. However, the same proof can be performed for arbitrary $d$. In this case, the subvariety $N\simeq \Pic^{d-\delta}_V(\calC^\nu)$ is a priori only a torsor over $\Pic^0_V(\calC^\nu)$.
	
\end{remark}

\subsection{Subfibration by Prym varieties}\label{SectionSubfibPrymVar}

In this section, we go further and construct a Lagrangian subfibration $P\subset N$ whose fibres consist of Prym varieties.

We use the notation of \S \ref{SectionCompactifiedJacobianLagrangianFibrations}. Let $p\colon S\to \matP^2$ be a double cover ramified in a general sextic $D_6\subset \matP^2$ (in particular $\Pic(S)=\matZ H_S$), denote $\iota\colon S\to S$ the (anti-symplectic) covering involution and $\Gamma\subset S$ the ramification locus. Fix some $k\in\matZ_{>0}$, $k\neq 6$ and set $\delta\coloneqq 3k$. We assume the following:
\vspace{6pt}

{\itshape 
	There exists a smooth locally closed subset $\calV\subset |kH_{\matP^2}|$ of codimension $\delta$ parametrizing smooth plane curves of degree $k$ which intersect $D_6$ in $\delta$ points of multiplicity $2$. 
}
\vspace{6pt}

\noindent This assumption holds for $k=5$, as proved \cref{LemmaLinearSystemD515TangD6isdim5}.

\begin{lemma}\label{LemmaMult2IntImpliesNodalOnK3}
	For each curve $D\in\calV$, the preimage $C=p^{-1}(D)$ is integral $\delta$-nodal, with nodes lying on $\Gamma$.
\end{lemma}

\begin{proof}
	The composition of the normalization $C^\nu\to C$ and $C\to D$ is a double cover, possibly ramified over $D\cap D_6$. In particular, the singular points of $C$ lie over $\Gamma$, and their preimages in $C^\nu$ consist of $1$ or $2$ points. In order to prove that these singular points are ordinary double points, it is enough to prove that the tangent cone of $C$ at a singular point is smooth and spans the tangent space of $S$. Since the statement about the type of singularities is local, we replace $\matP^2$ with $\matA^2$ with coordinates $x,y$, and we assume that $D=V(g)$, $D_6=V(f)$ for $g$ and $f$ polynomials in $x,y$ of degree $k$ and $6$ respectively. Given $p\in D\cap D_6$, we can assume $p=(0,0)\in D\cap D_6$, and locally $f=y- a_2x^2-\cdots -a_6x^6$ and $g=y- b_2x^2-\cdots -b_kx^k$, with $a_2\neq b_2$. The double cover $S$ is locally given by $z^2=f$ in $\matA^3$ with coordinates $x,y,z$. Hence, $C$ is locally given by the ideal 
	$$I=(y-\sum_i b_ix^i,z^2-y+\sum_i a_ix^i)=(y-\sum_i b_ix^i, z^2+\sum_i(a_i-b_i)x^i).$$
	The tangent cone of $C$ at $(0,0,0)$ is then given by the ideal of initial terms of polynomials in $I$, which turns out to be $J=(y,z^2+(b_2-a_2)x^2)$. We conclude that the singularity of $C$ at $(0,0,0)$ is an ordinary double point.
	
	To conclude, we need to prove that $C$ is irreducible. Since $D$ is irreducible, any subcurve of $C$ dominates $D$. Since $C\to D$ is a double cover, $C$ has at most two components. Assume for the sake of contradiction that $C=C_1\cup C_2$. Each component $C_i$ admits a degree $1$ map to $D$ and is smooth outide of its intersection $Z=\Gamma\cap C_i$ with $\Gamma$. Since $C$ admits only  $A_1$-singularities on $Z$, $C_1$ and $C_2$ must be smooth and intersecting transversally on $\Gamma$. In particular, $C_1$ and $C_2$ are isomorphic to $D$, and therefore have genus $\frac{1}{2}(k^2-3k)+1$. Since $\Pic(S)=\matZ H_S$, we have $C_i\in |(k/2)H_S|$ and thus $g(C_i)=k^2/4+1$. The only possibility is $k=6$, which we ruled out in the assumption.
\end{proof}

\begin{remark}
The proof of \cref{LemmaMult2IntImpliesNodalOnK3} reveals that higher multiplicity intersections of $D_6\cap D$ (i.e. $a_2=b_2$) lead to worse singularities on $S$. 
\end{remark}

In the following, we identify $\calV$ with its image $p^*\calV\subset p^*|kH_{\matP^2}|\subset B\coloneqq |kH_S|$. Let $\calC\to \calV$ be the family of nodal curves,  $\calC^\nu \to \calV$ be its fibrewise normalization (\cref{LemmaFibrewiseNormalizationCurves}), and $\calD\to \calV$ be the family of smooth plane curves of degree $k$. The composition $\calC^\nu \to \calD$ is a family of \'etale double covers over $\calV$, leading to a relative Prym variety $\Prym_\calV(\calC^\nu/\calD)$ defined as the identity component of the fixed locus of $\tau_{\calC^\nu}\coloneqq -\iota^*$, where the morphism $-1$ (well defined since $\Pic^0_\calV(\calC^\nu)$ is a group scheme) is a relative version of $(-)^\vee$. We refer to \cite[\S 4.1]{LazaSaccaVoisinHKCompactificationIntermediateJacobianFibrationCubic4fold} for more details about relative Prym varieties.

Finally, we set $M\coloneqq M_S(0,kH_S,3k-k^2)\to B$ the Beauville--Mukai system. Note that the curves $C\in |kH_S|$ have genus $p_a\coloneqq k^2+1$.
Following \cite{ASFRelativePrymVar}, the involution $\iota$ induces an anti-symplectic birational map
$$\iota^*\colon M \dashrightarrow M, F\mapsto \iota^*F$$
which is defined (and an isomorphism) when restricted to $M_{\calV}$. 
Again in \textit{loc. cit.}, \S $3.5$, the authors construct a birational map
$$\kappa\colon M\dashrightarrow M, F \mapsto \calE xt^1(F,\calO_S(\Gamma-kH_S))$$
which also restricts to an automorphism of $M_\calV$ (a direct computation 
shows that the Mukai vector $(0,kH_S,3k-k^2)$ is indeed preserved). The map $\kappa$ plays the role of a relative (twisted) dual for the torsion sheaves. Indeed, since any $C\in \calV$ has only planar singularities, it is Gorenstein, in particular $C$ admits a dualizing sheaf $\omega_C\simeq i_C^*\calO_S(kH_S)$ for $i_C\colon C\hookrightarrow S$ the closed immersion, see \cite[Chap. II, section $1$]{BarthPetersVandeVenCompactComplexSurfaces}. As proved in \cite[Lemma 3.7]{ASFRelativePrymVar}, there is an isomorphism
\begin{eqnarray}\label{EqnExt1HomTwistedDual}
\calE xt^1\left((i_C)_*F,\calO_S(\Gamma-kH_S)\right)\simeq (i_C)_*\calH om(F,i_C^*\Gamma).
\end{eqnarray}

\begin{lemma}[\cite{ASFRelativePrymVar}, \S $3.3$]
	In the notation of \cref{ThmSubfibrationJacobianNormalizations}, the map $\tau\coloneqq\kappa\circ \iota$ restricts to an involution
	$$\tau_N\colon N \xrightarrow{\sim} N.$$
\end{lemma}

\begin{proof}
	First we prove that $\tau$ preserves $N$. Pick a sheaf $F\in N$, so that $F$ is a torsion-free rank $1$ sheaf supported on a $\delta$-nodal curve $C\in V$ (recall that $V$ denotes some component of the Severi variety of $\delta$-nodal curves), such that $F$ is non-locally free over each of the $\delta$ nodes of $C$. Then $\iota^*F$ is supported on $\iota(C)$ (which is also $\delta$-nodal), and it must be non-locally free over each node of $\iota(C)$, hence $\iota^*F\in N$. 
	
	On the other hand, from (\ref{EqnExt1HomTwistedDual}) the sheaf $\calE xt^1(F,\calO_S(\Gamma-C))$ is still a torsion-free rank $1$ sheaf of same degree supported on the same curve $C$, so $\kappa$ restricts to an isomorphism $\kappa_C\colon \CPic^\delta(C) \xrightarrow{\sim} \CPic^\delta(C)$. Therefore $\kappa_C$ must preserves the only closed orbit $O=\nu_*\Pic^0(C^\nu)\subset \CPic^{\delta}(C)$, as this orbit consists of the only points where the tangent space of $\CPic^{\delta}(C)$ has maximal dimension $p_a+\delta$ (in other words, this orbit is the "most singular" part of $\CPic^{\delta}(C)$). We conclude that $\kappa$ preserves $N$ as well.
	
	Finally, the property $\tau^2(F)=F$ is a consequence of the fact that $\kappa^2(F)=F$ (a pure dimension $1$ sheaf on a surface is always reflexive \cite[Prop. 1.1.10]{HuybrechtsLehnModuliofsheaves}) and that $\kappa$ and $\iota$ commute (local Hom and flat pullback commute).
\end{proof}

Let $\Fix^0(N)$ be the connected component of the fixed locus in $N$ of the action of $\tau_N$ that contains the image of the $0$-section of $\Pic^0_\calV(\calC^\nu)\to \calV$. Denote $\Fix^0(N)_\calV=\Fix^0(N)\cap \pi^{-1}(\calV)$ the restriction over $\calV$.

\begin{theorem}\label{ThmPrymSubfibrationOverNodalCurves}
	The locus $\calV$ is an open subset of $\pi(\Fix^0(N))$. The map $\Upsilon\colon \Pic^0_V(\calC^\nu) \xrightarrow{\sim} N$ restrict to an isomorphism $\Prym_\calV(\calC^\nu/\calD) \xrightarrow{\sim} P\coloneqq \Fix^0(N)_\calV$.
	
	Moreover, the restriction of $\sigma_N$ to $P$ is still non-degenerate. In particular, $P$ is a smooth symplectic subvariety of $N$ and $\pi_P\colon P \to \calV$ describes a Lagrangian subfibration of $\pi_N\colon N\to V$.
\end{theorem}

\begin{proof}

	\textbf{First step:} By \cite[Prop. 3.11]{ASFRelativePrymVar}, $\tau$ is symplectic, hence its restriction $\tau_N$ to $N$ is also symplectic. Since any component of the fixed locus of a symplectic involution on a smooth symplectic variety is again symplectic and smooth, we conclude that $\Fix^0(N)$ is symplectic (with respect to the symplectic structure induced from $M$) and smooth.
	
	\textbf{Second step:} We prove that $\Prym_\calV(\calC^\nu/\calD)$ lands by $\Upsilon$ in $\Fix^0(N)$. Since $\Prym_\calV(\calC^\nu/\calD)$ is irreducible and contains the $0$-section of $\Pic^0_\calV(\calC^\nu)$, it suffices to proves that it lands in the fixed locus of $\tau$. Pick $C\in \calV$, $\calL\in \Prym(C^\nu/D)$ where $D\coloneqq p(C)$ is the image of the curve in $\matP^2$. Denote $\nu\colon C^\nu \to C$ the normalization and $x_i$, $y_i$ the preimages by $\nu$ of the nodes $z_i\in C$, $i=1,\dots,\delta$. Denote $i_C\colon C\hookrightarrow S$ the closed immersion, and $\mathsf{n}\coloneqq i_C\circ\nu\colon C^\nu \to S$ the composition. We know that $i_C$ is $\iota$-equivariant, therefore it suffices to prove 
	$$i_C\circ \tau_C=\tau\circ i_C.$$
	This is the content of the following refinement of (\ref{EqnExt1HomTwistedDual}).
	\begin{lemma}\label{LemmaExt1IsPushforwardHom}
		We have $$\calE xt^1(\mathsf{n}_*\calL,\calO_S(\Gamma-C)) \simeq \mathsf{n}_*\calH om(\calL,\calO_{C^\nu})$$
	\end{lemma}

	\begin{proof}
	Recall that the global sections of $\omega_C$ correspond to the pushforward of $1$-forms on $C^\nu$ with poles over the preimages of the nodes, and opposite residues in the preimages of each node.  Pushforward by $\nu$ gives an exact sequence
	$$0 \to \nu_*\omega_{C^\nu} \to \nu_*\omega_{C^\nu}(\sum_i(x_i+y_i)) \xrightarrow{r}  \bigoplus_i \matC(z_i)^{\oplus 2} \to 0$$
	and $\omega_C$ is given as the preimage by $r$ of the anti-diagonal of $\bigoplus_i\matC(z_i)^{\oplus 2}$. In particular there is a short exact sequence
	\begin{eqnarray}\label{EqnExSeqDualSheafNodalCurve}
	0 \to \nu_*\omega_{C^\nu} \to \omega_C \to \bigoplus_i\matC(z_i) \to 0.
	\end{eqnarray}
	Pulling back by $\nu$, we get the exact sequence
	$$0 \to (\nu^*\nu_*\omega_{C^\nu})_{f} \to \nu^*\omega_C \to \bigoplus_i \matC(x_i)\oplus \matC(y_i) \to 0,$$
	where the first term is the locally free quotient of $\nu^*\nu_*\omega_{C^\nu}$. It is easy to see that $(\nu^*\nu_*\omega_{C^\nu})_{f}\simeq \omega_{C^\nu}$, hence $\nu^*\omega_C\simeq \omega_{C^\nu}\otimes\calO_{C^\nu}(\sum_ix_i+y_i)$.
	
	Using Grothendieck-Verdier duality \cite[Theorem 3.34]{HuybrechtsFMTransform}, we get
	\begin{eqnarray*}
		\mathsf{n}_*\calH om(\calL,\calO_{C^\nu}) &\simeq&  \mathsf{n}_*\calH om(\calL,\omega_{C^\nu}\otimes \mathsf{n}^*\calO_S(\Gamma-C)) \\
		&\simeq & \calE xt^1(\mathsf{n}_*\calL,\calO_S(\Gamma-C)),
	\end{eqnarray*}
	where we use that $\mathsf{n}^*\calO_S(\Gamma)\simeq \calO_{C^\nu}(\sum_i x_i+y_i)$ and $i_C^*\calO_S(C)\simeq \omega_C$.
	\end{proof}

%
	
		\textbf{Third step:} Denote $T=\pi(\Fix^0(N))\subset V$ and $P=\Upsilon(\Prym(C^\nu/D))$. We prove that $\calV\subset T$ is an open subset. Since clearly $\Fix^0(N)_{[C]}=\Prym(C^\nu/D)$ for any $C\in\calV$ (it follows from the definition of $\tau$ and $\Prym(C^\nu/D)$), it implies that $P=\Fix^0(N)_\calV$.
		
		Step $1$ shows that $\calV\subset T$, so we just need to prove $\dim(\calV)=\dim(T)$. Denote $v=\dim(\calV)=(1/2)(k^2-3k)$ and $t=\dim(T)$. The fiber of $\Fix^0(N)$ over a point $[C]\in \calV$ is $\Prym(C^\nu/D)$, which is $m$-dimensional. There is an open $\calW\subset \calV$ around $[C]$ over which $\pi\colon \Fix^0(N)_\calW \to \calW$ is smooth, in particular $\Fix^0(N)_\calW$ (hence $\Fix^0(N)$) has dimension $t+v$. Since $\pi$ is Lagrangian, any $d\pi$-lift of $T_{T,[C]}$ to $T_{\Fix^0(N),[\calO_{C^\nu}]}$ is isotropic (in particular, its dimension is at most $(v+t)/2$) and have dimension $t$ (by smoothness of $\Fix^0(N)_\calW \to \calW)$. Therefore we have $v\leq t\leq (1/2)(v+t)$, so $v=t$.
	
	\textbf{Fourth step:} We know by Step $3$ that $P$ is smooth. Since $\Prym_\calV(\calC^\nu/\calD)$ is smooth over $\calV$ \cite[Thm. 4.33]{LazaSaccaVoisinHKCompactificationIntermediateJacobianFibrationCubic4fold}, by miracle flatness the map $\Upsilon_P\colon \Prym_\calV(\calC^\nu/\calD) \to P$ is an isomorphism.
	
	Finally, the map $\pi_P\colon P \to \calV$ is automatically a Lagrangian fibration because $\pi_N\colon N\to V$ is so and $\calV$ has the same dimension as the fibres of $\pi_P$.
\end{proof}



\begin{remark}\label{RmkSumUpConstructionsFibrations}
	Let us sum-up here the constructions in \S \ref{SectionCompactifiedJacobianLagrangianFibrations}. For $p\colon S\to\matP^2$ a general K3 surface of degree $2$ as in \S \ref{SectionSubfibPrymVar} we have the following commutative diagram:
	
	\begin{center}
	\begin{tikzcd}
		P \ar[d,"\pi_P"] \ar[r,hook] & N \ar[d,"\pi_N"] \ar[r,hook] & M \ar[d,"\pi"] \\
		\calV \ar[r,hook] &  V \ar[r,hook] & {|} kH_S{|}	
	\end{tikzcd}
\end{center}
where the inclusions are locally closed immersions and $\pi_P$ and $\pi_N$ are Lagrangian subfibrations of $\pi$. The subset $V$ consists of nodal curves, while $\calV$ consists of nodal curves with nodes over the ramification locus $\Gamma\subset S$. We have $\dim(|kH_S|)=k^2+1$, $\dim(V)=k^2-3k+1$ and $\dim(\calV)= (1/2)(k^2-3k)$. For $[C]\in \calV$ with $C^\nu\to C$ its normalization and $p(C)=D\subset \matP^2$, the fibres are described as follow:
$$P_{[C]}\simeq \Prym(C^\nu/D) \subset N_{[C]}\simeq \Pic^0(C^\nu) \subset M_{[C]}\simeq \CPic^0(C).$$

\end{remark}

\begin{remark}
	In this lengthy remark, we aim to describe (pointwise) the natural pairing between $T_\calV$ and $\Omega^1_{P/\calV}$ induced by the symplectic form $\sigma_P$.

	Given a point $F\in \CPic^\delta(C)$, the natural pairing between $T_{V,[C]}$ and $T_{\Pic^0(C^\nu),[\calO_{C^\nu}]}\subset T_{\CPic^\delta(C),[F]}$ induced by $\sigma_N$ deforms to a natural pairing between $T_{V,[C]}$ and $T_{\Pic^0(C^\nu),[\calO_{C^\nu}]} \subset T_{\CPic^\delta(C),[\calL]}$ for a line bundle $\calL\in\Pic^\delta(C)\subset\CPic^\delta(C)$ close to $F$. 
	Serre duality for line bundles gives $T^\vee_{\CPic^\delta(C),[\calL]}\simeq H^0(C,\omega_C)$.
	Computing the cohomology sequence induced by (\ref{EqnExSeqDualSheafNodalCurve}), we get an exact sequence
	\begin{eqnarray}\label{EqExSeqGlobalSectionDualSheafNodalCurve}
	0 \to H^0(C^\nu,\omega_{C^\nu}) \to H^0(C,\omega_C) \to \bigoplus_i \matC \to 0.
	\end{eqnarray}
	In particular, the inclusion in (\ref{EqExSeqGlobalSectionDualSheafNodalCurve}) is another interpretation of (\ref{EqnDecompTangentSpaceCompJac}) for $m=0$.
	
	Since $\calL$ is a smooth point of the fibre $\CPic^\delta(C)=M_{[C]}$, contraction with $\sigma_M$ induces a natural isomorphism
	\begin{eqnarray}\label{EqContractionSmoothPointCompJac}
		\calA_\calL\colon H^0(S,\calO_S(k))/\matC s \simeq T_{B,[C]} \xrightarrow{\sim} T^\vee_{\CPic^\delta(C),[\calL]}\simeq H^0(C,\omega_C)
	\end{eqnarray}
	which is given by the composition of $H^0(S,\calO_S(k))\to H^0(C,\calO_C(k))$ and the isomorphism $H^0(C,\calO_C(C))\xrightarrow{\sim} H^0(C,\omega_C)$ \cite[II.1]{BarthPetersVandeVenCompactComplexSurfaces}. In other words, $\calA_\calL$ maps the class of a curve $C'$ to the divisor on $C$ associated to $C\cap C'$.
	
	One can check that $\calA_\calL$ maps isomorphically $T_{V,[C]}$ to $H^0(C^\nu,\omega_{C^\nu})$ (this is essentially \cref{ThmSubfibrationJacobianNormalizations}). Denoting $D=p(C)$, the subspace $T^\vee_{\Prym(C^\nu/D),[\calO_{C^\nu}]}$ identifies with the quotient $H^0(C^\nu,\omega_{C^\nu})/(p\circ\nu)^*H^0(D,\omega_D)$. We claim that $T_{V,[C]}$ decomposes as a direct sum 
	$$T_{V,[C]}=T_{\calV,[C]}\oplus W,$$
	and $\calA_\calL(W)=(p\circ\nu)^*H^0(D,\omega_D)$. Let us explain why this holds.
		
	The linear system $B=|kH_S|=\matP H^0(S,\calO_S(k))$ admits two supplementary subspaces, namely $H^0(S,\calO_S(k))=W_1\oplus W_2$ with $W_1=p^*H^0(\matP^2,\calO_{\matP^2}(k))$ and $W_2=\gamma\otimes H^0(\matP^2,\calO_{\matP^2}(k-3))$, where $\gamma$ is a global section of $H^0(S,\calO_S(3))$ whose zero locus is $\Gamma$. The spaces $W_1$ and $W_2$ are the eigenspaces of the natural action of $\iota$ on $H^0(S,\calO_S(k))$. By construction, $\calV \subset W_1$, and since $\calA_\calL$ maps $W_2$ to $H^0(C^\nu,\omega_{C^\nu})$, we get $W_2\subset T_{V,[C]}$. A dimension count give $T_{V,[C]}\simeq T_{\calV,[C]}\oplus W_2.$
	By definition, $\calA_\calL(W_2)=(p\circ\nu)^*H^0(D,\omega_D)$. Taking quotients, we see that $\calA_\calL$ induces an isomorphism
	$$T_{\calV,[C]} \xrightarrow{\sim} T^\vee_{\Prym(C^\nu/D),[\calO_{C^\nu}]}.$$

\end{remark}

\section{Cubic fourfold containing a plane: geometric setup}\label{SectionGeomSettings}

The goal of this section is to introduce the geometric setup for the statement and proof of Theorem \ref{ThmMainPrymSubFibration}.
For a proof of all the claims, we refer to \cite[\S $5.1.3$ and \S $6.1.2$]{HuybrechtsCubicHypersurfaces}. Let $X$ be a complex cubic fourfold in $\matP^5$ containing a plane $P\subset X$, and assume $(X,P)$ is generic (see \S \ref{SectionGeneralityAssumptions}). Consider the projection from $P$
$$\pi_X\colon\Bl_P(X)\to \matP^2$$
to a generic $\matP^2\subset \matP^5$. The fibre over $y\in\matP^2$ is the residual quadric $Q_y\subset X$ of the intersection $\overline{yP}\cap X=Q_y\cup P$.
The discriminant curve $D_6\subset\matP^2$ of singular quadrics is a sextic plane curve. For $(X,P)$ general, $D_6$ is smooth, and the only singular quadrics are cones over a conic. Denoting $\tau_X\colon\Bl_P(X) \to X$ the blow-up map and $E$ the exceptional divisor, the morphism $\pi_X$ is associated to the linear system $|\tau_X^*\calO_X(1)\otimes \calO_{\Bl_P(X)}(-E)|$.


Attached to the family of quadric surfaces, there is a relative Fano variety of lines
$\widetilde{F} \to \matP^2$
which splits (from its Stein factorization)
$$\widetilde{F} \xrightarrow{\alpha} S \xrightarrow{\pi} \matP^2,$$
where $\pi\colon S \to \matP^2$ is a double cover ($K3$ surface of degree $2$) branched over $D_6$, and $\alpha\colon \widetilde{F}\to S$ is a morphism with $\matP^1$ fibres (a \textit{Brauer-Severi} variety). We denote $H_{\matP^2}$ the divisor class of a line in $\matP^2$ and $H_S\coloneqq \pi^{-1}H_{\matP^2}$ the primitive polarization on $S$ of square $2$. 
\newline

Consider the open subset $\calU\subset (\matP^5)^\vee$ of hyperplanes as in \S \ref{SectionGeneralityAssumptions}. Pick $H\in\calU$ and denote $Y=X\cap H$ (a smooth cubic threefold) and $L=H\cap P\subset Y$. The projection $\pi_X$ restricts to a map
$$\pi_Y \colon  \Bl_L(Y) \to \matP^2,$$
where we identify the strict transform $\widetilde{Y}\subset \Bl_P(X)$ with $\Bl_L(Y)$. Denote $E'$ the exceptional divisor of the blow-up map $\tau_Y\colon \Bl_L(Y)\to Y$. Since the restriction of $\calO_{\Bl_P(X)}(-E)$ to $\Bl_L(Y)$ is $\calO_{\Bl_L(Y)}(-E')$, the morphism $\pi_Y$ is associated to the linear system $|\tau_Y^*\calO_Y(1)\otimes \calO_{\Bl_L(Y)}(-E')|$. The fibres of $\pi_Y$ are the intersection of the fibres of $\pi_X$ with $H$. The general fibres are smooth conics, the discriminant curve $D_H\subset\matP^2$ is a smooth plane quintic and the fibres over it are the reduced union of two secant lines.

Consider the curve $C_H=\overline{\{L'\in F(Y)\mysetminus L \ | \ L\cap L'\neq \emptyset\}}$. It is smooth of genus $11$. Lines parametrized by $C_H$ are contained in the fibres of $\pi_Y$ (in particular, in the fibres over $D_H$), $C_H\to D_H$ is an étale double cover, and the intermediate Jacobian of $Y$ can be described as
$$J(Y)\simeq \Prym(C_H/D_H).$$

Following \cite{LazaSaccaVoisinHKCompactificationIntermediateJacobianFibrationCubic4fold}, there exists a universal family of hyperplane sections $\overline{\calY} \to (\matP^5)^\vee$ and we can consider its restriction $\calY \to \calU$ to the open subset $\calU\subset (\matP^5)^\vee$.
Working in relative settings, we construct in \S \ref{SectionConstructionOfRelativeCurve} a family of étale double covers $\calC \to \calD$ over $\calU$, whose fibre over $H\in \calU$ is the double cover $C_H \to D_H$. By \cite[\S 5]{LazaSaccaVoisinHKCompactificationIntermediateJacobianFibrationCubic4fold}, the intermediate Jacobian fibration $\calJ \to \calU$ can be identified with the relative $\Prym$
$$\calJ \simeq \Prym_\calU(\calC/\calD) \to \calU.$$

\noindent and it can be equipped with a holomorphic symplectic structure.

\begin{remark}
	Note that the authors in \cite{LazaSaccaVoisinHKCompactificationIntermediateJacobianFibrationCubic4fold} assume $X$ to be general in order to construct a smooth compactification of $\Prym_\calU(\calC/\calD)$. In particular, their construction does not work for cubics containing a plane, however the relative Prym construction and its isomorphism with the Jacobian fibration still holds over the open $\calU$ that we consider, thanks to the existence of the relative curve $\calC$. Sacc\`a \cite{SaccaBiratGeomIntJacFibCubic4f} constructed a smooth compactification for \textit{all} cubic fourfolds. See \S \ref{SectionQuestionMainThm} for more remarks.
	\end{remark}



As an application of the construction developped in \S \ref{SectionSubfibPrymVar}, the following theorem describes the relative Prym fibration as a subfibration (up to a cover) of a Beauville--Mukai system on the $K3$ surface $S$ associated to $X$.

\begin{theorem}\label{ThmMainPrymSubFibration}
There is a commutative diagram 
\begin{center}
\begin{tikzcd}
	\Prym_\calU(\calC/\calD) \ar[r,"f"] \ar[d] & M\coloneqq M_S(0,5H_S,-10) \ar[d] \\
	\calU \ar[r,"i"] & {|}5H_S{|}
\end{tikzcd}
\end{center}
such that
\begin{enumerate}[label=(\roman*)]
	\item the vertical arrows are Lagrangian fibrations,
	\item the horizontal arrows are quasi-finite,
	\item the image $f(\Prym_\calU(\calC/\calD))$ is symplectic with respect to the restriction of the symplectic structure of $M$, so that the maps $f$ and $i$ realize $\Prym_\calU(\calC/\calD)$ as a finite cover of a Lagrangian subfibration of $M\to |5H_S|$.
\end{enumerate}

\end{theorem}

We prove \cref{ThmMainPrymSubFibration} at the end of \S \ref{sectionConstructionf}.

\subsection{In coordinates}\label{SectionCoordinates}

A description of the equations of $D_H$ and $D_6$ with respect to the ones of $X$ can be found in \cite[\S 5]{AuelCollThParimUnivUnramCohCubic4fPlane}. We set $X\subset \matP^5=\matP V_6$, with homogeneous coordinates $(x:y:z:t_0:t_1:t_2)$, and $P=\matP V_3\subset X$. We can assume that $P=\{x=y=z=0\}$ and therefore the equation of $X$ has the shape
\begin{equation}\label{EqCubicContainingPlane}
\sum_{0\leq i<j\leq 2}a_{ij}t_it_j+\sum_{0\leq i \leq 2} b_it_i^2+c=0
\end{equation}
where $a_{ij}$, $b_i$ and $c$ are homogeneous polynomials in $x,y,z$ of degree $1$, $2$ and $3$ respectively.
The projection $\pi_X\colon\Bl_P(\matP^5)\to \matP^2\coloneqq \matP(V_6/V_3)$ is identified with the projective bundle $\matP\calE \to \matP^2$, where $\calE=(V_3\otimes\calO_{\matP^2})\oplus \calO_{\matP^2}(-1)$. The quadric bundle $\Bl_P(X)\to\matP^2$ is then defined by the quadratic form $q_X\colon\calE \to \calO_{\matP^2}(1)$ given by
$$q(t_0,t_1,t_2,s)=\sum_{i,j}a_{ij}t_it_j+\sum_i b_it_i^2s+cs^2,$$
where $s$ is a local equation for $\calO_{\matP^2}(-1)$. See \cite{AuelBernardaraBolognesiFibQuadricsCliffalg} for a general treatment of quadratic forms and quadric bundles. The Gram matrix associated to (the bilinear form associated to) $q$ is given by
\begin{equation}\label{EqMatrixCubics}
M_q\coloneqq\begin{pmatrix}
	2a_{00} & a_{01} & a_{02} & b_0 \\
	a_{01} & 2a_{11} & a_{12} & b_1 \\
	a_{02} & a_{12} & 2a_{22} & b_2 \\
	b_0 & b_1 & b_2 & 2c 
\end{pmatrix}
\end{equation}
and the discriminant curve $D_6$ is given by the determinant of this matrix. Now pick $H\in\calU$, set $L=P\cap H=\matP V_2$ with $V_2\subset V_3$ and denote $\calE_H=(V_2\otimes\calO_{\matP^2})\oplus \calO_{\matP^2}(-1)$. We can assume that the equation of $H$ gives a linear relation $t_0=l(t_1,t_2)$, and the conic bundle $\Bl_L(Y)\to\matP^2$ is given by the quadratic form $q_H\colon\calE_H\to\calO_{\matP^2}(1)$ defined by $q_H=q(l(t_1,t_2),t_1,t_2,s)$. Again, the discriminant curve $D_H$ is defined by the determinant of the Gram matrix of $q_H$. For instance, when $H=\{t_0=0\}$, the determinant of $q_H$ is the first minor $M_{11}$ of $M_q$.

\subsection{Generality assumptions}\label{SectionGeneralityAssumptions}

In view of the rest of the paper, we want to pick a general cubic $X$ containing a plane $P$ and an open $\calU\subset (\matP^5)^\vee$ of hyperplanes $H$ for which

\begin{enumerate}[label=(\roman*)]
	\item $X$ and $D_6$ are smooth, \label{Cond1}
	\item $D_6$ does not admit a tritangent (see \cref{LemmaLinearSystemD515TangD6isdim5}),\label{Cond2}
	\item $Y$, $C_H$ and $D_H$ are smooth, \label{Cond3}
	\item $D_H$ and $D_6$ intersect with multiplicity at most $2$ (see \cref{LemmaIntersectionD6DHmult2}). \label{Cond4}
\end{enumerate}

Consider the subspace $\calS\subset |\calO_{\matP^5}(3)|\times \matG(2,\matP^5)\times (\matP^5)^\vee$ of triples $(X,P,H)$ with $P\subset X$.
We will prove that the conditions \ref{Cond3} and \ref{Cond4} are open in $\calS$ and we will find explicit examples $(X,P,H)$ satisfying each of them. An important remark (which simplifies the computations) is that we do not require this example to satisfy all the conditions at the same time. In particular, we do not need $X$ or $Y$ to be smooth in order to study $D_H$ and $D_6$, that we define as in \S \ref{SectionCoordinates}.

\begin{enumerate}[wide, labelwidth=!, labelindent=\labelindent, label=(\roman*)]
	
	\item[\textit{Conditions} \ref{Cond1} \textit{and} \ref{Cond2}:] For general $(X,P)$ the curve $D_6$ is smooth \cite[Remark 1.5]{HuybrechtsCubicHypersurfaces}. Moreover, the number of pairs $(X',P')$ with associated sextic $D_6$ is finite. Indeed, in this case the associated K3 of $X'$ is also $S$ (the double cover of $\matP^2$ ramified over $D_6$), and the Kuznetsov component $\Ku(X')$ of $\D^b(X')$ satisfies $\Ku(X')\simeq \D^b(S,\alpha')$ for some $\alpha'\in\Br(S)[2]$ \cite[Prop. 4.3]{KuznetsovDerivedCatCubicFourfolds}. But $\Br(S)[2]$ is finite, and there exists finitely many $X'$ with $\Ku(X)\simeq \Ku(X')$ \cite[Thm. 1.1]{HuybrechtsK3CatCubicFourfold}. Since the moduli space of smooth cubic fourfolds containing a plane has the same dimension as the moduli space of smooth plane sextics ($=19$), for general $(X,P)$ we can assume that $D_6$ is general, in particular it admits no tritangent.
	\newline

	\item[\textit{Condition} \ref{Cond3}:] A general hyperplane section of a smooth cubic fourfold is smooth. The equation of $D_H$ depends polynomially on the equation of $X$ and $H$ (see \S \ref{SectionCoordinates}) so its smoothness is an open condition in $\calS$. See \cref{ExXPHwithgoodproperties}, \cref{ExSmoothDH} for an example of smooth $D_H$. 
	
	  Recall that if $T$ is a cubic hypersurface given in $\matP^n$ by the equation $F$, then a line $L\subset T$ is of the first type in $T$ if the restriction $\partial_iF|_L$ of the partial derivatives of $F$ to $L$ generates $\calO_L(2)$. When $D_H$ is smooth, $C_H$ is smooth whenever $L$ is of the first type in $Y$ \cite[Cor. 1.23]{HuybrechtsCubicHypersurfaces}. Note that $L\subset P$ is of the first type in $Y$ if and only if it is so in $X$. Indeed, we can assume that $X$, $P$ and $H$ are given as in \S \ref{SectionCoordinates}. Therefore, the equation $F=F(x,y,z,t_0,t_1,t_2)$ of $X$ is given by (\ref{EqCubicContainingPlane}), the equation of $H$ is given by $t_0=0$, the equation of $Y$ is given by $F_H=F(x,y,z,0,t_1,t_2)$ and $\calO_L(2)$ is spanned by $\{t_1^2,t_2^2,t_1t_2\}$. Therefore the difference $F-F_Y$ is divisible by $t_0$, and hence no monomial in $F-F_Y$ is divisible by $t_1^2$, $t_2^2$ nor $t_1t_2$. We obtain that the partial derivatives of $F$ and $F_Y$ generates the same sub-vector space of $\calO_L(2)$, as claimed. 
	 
	 From this remark we also see that $L$ being of the first type in $X$ is an open condition in $\calS$ (in fact, the relative locus $F_2(\calX)\subset |\calO_{\matP^5}(3)|\times \matG(1,\matP^5)$ of lines of the second type is a proper closed subscheme). See \cref{ExXPHwithgoodproperties}, \cref{ExLFirstType} of an example with $L$ of the first type. 

  \begin{remark}
      Geometrically, a line $L\subset Y\subset \matP^4$ is of the second type if and only if there exists a plane $L\subset P_L\subset \matP^4$ for which $P_L\cap Y=2L+L'$ for some residual line $L'$. Avoiding this phenomenon is an important step in the construction of a compactification of $\calJ$, see \cite[\S 2]{LazaSaccaVoisinHKCompactificationIntermediateJacobianFibrationCubic4fold}.
  \end{remark}
	 
	 \item[\textit{Condition} \ref{Cond4}:] 
	 It is enough to prove that the locus in $|\calO_{\matP^2}(5)|\times |\calO_{\matP^2}(6)|$ of curves which intersect with at least $1$ point of multiplicity $3$ is closed. Denote $\calC_i$ the universal curve over $|\calO_{\matP^2}(i)|$,  $i=5,6$. Consider the fibre product
	 $$\begin{tikzcd}
	 \calC_{5,6} \ar[r] \arrow[dr, phantom, "\lrcorner", very near start] \ar[d] & \calC_5 \ar[d] \\
	 \calC_6 \ar[r] & \matP^2.
	 \end{tikzcd}$$
	 The fibre of the natural map $\calC_{5,6}\to |\calO_{\matP^2}(5)|\times |\calO_{\matP^2}(6)|$ over some pair $(C_5,C_6)$ is the intersection $C_5\cap C_6$, seen as a subscheme of $\matP^2$. Consider the open $\calU_{5,6}\subset |\calO_{\matP^2}(5)\times \calO_{\matP^2}(6)|$ of curves $(C_5,C_6)$ such that $C_5\cap C_6$ has dimension $0$ (it is open by \cite[Cor. 14.113]{GortzWedhornAlgGeoSchemes}), and denote again $\calC_{5,6}$ the restriction to this open. Hence $\calC_{5,6}\to \calU_{5,6}$ is a family of $0$-dimensional subschemes of $\matP^2$ of length $30$. The universal property of Hilbert schemes gives a map
	 $$\calU_{5,6} \to (\matP^2)^{[30]}$$
	 and the locus of curves with at least $1$ intersection point of multiplicity $3$ is the preimage of the closed subscheme of $(\matP^2)^{[30]}$ consisting of subschemes with at least $1$ triple point. The latter is indeed closed since it is itself the preimage by the Hilbert-Chow morphism of the locus in $(\matP^2)^{(30)}$ of $30$-uples with a point appearing at least $3$ times.
	 
	 To conclude, it is enough to find one example where $D_6$ and $D_H$ intersects with multiplicity at most $2$ everywhere, see \cref{ExXPHwithgoodproperties}, \cref{ExDHD6IntMult2}.
	 
\end{enumerate}
	
	\begin{example}\label{ExXPHwithgoodproperties}
		We use the notation of \S \ref{SectionCoordinates}.
		\begin{enumerate}[wide, labelwidth=!, labelindent=\labelindent]
			\item Consider the cubic fourfold with associated matrix (\ref{EqMatrixCubics}) given by 
			$$\begin{pmatrix}
			2a_{00} & a_{01} & a_{02} & b_0 \\
			a_{01} & x+y+z & x & 0 \\
			a_{02} & x & x & z^2 \\
			b_0 & 0 & z^2 & x^3+y^3+z^3 
			\end{pmatrix} $$
			for any $a_{00},a_{01},a_{02},b_0$, and set $H=\{t_0=0\}$. Therefore $D_H$ is given by the first minor $M_{11}$ of this matrix. One can check that $D_H$ is smooth using the Jacobian criterion.
			 \label{ExSmoothDH}
			\item Consider the cubic fourfold given by the equation $xt_1^2+yt_2^2+zt_1t_2$ and set $H=\{t_0=0\}$. Clearly the partial derivatives along $x,y,z$ generate $H^0(L,\calO_L(2))=\Span(t_1^2,t_2^2,t_0t_1)$.
			\label{ExLFirstType}
			\item Consider the cubic fourfold with associated matrix
			$$\begin{pmatrix}
			0 & z-2y & z-2x & (x-y)(x-2y) \\
			z-2y & x & 0 & 0 \\
			z-2x & 0 & y & 0 \\
			(x-y)(x-2y) & 0 & 0 & x^3+y^3+z^3,
			\end{pmatrix} $$
			and set $H=\{t_0=0\}$. Denoting $F$ the Fermat cubic $x^3+y^3+z^3$, the curve $D_H$ is given by the equation $xyF$ (i.e. $D_H$ is the union of the Fermat cubic and two lines), and $D_6$ is given by the equation
			$$(z-2y)^2yF+(z-2x)^2xF+(x-y)^2(x-2y)^2xy.$$
			It is easy to check that $D_6$ and $D_H$ have no common component. For any $q\in \matP^2$, the multiplicity of the intersection is computed as follows:
			\begin{eqnarray}
			m_q(D_H,D_6) &=& m_q(xyF,(z-2y)^2yF+(z-2x)^2xF+(x-y)^2(x-2y)^2xy) \nonumber \\
			&=& 2m_q(x,z-2y)+2m_q(y,z-2x)+2m_q(F,x-y)+2m_q(F,x-2y) \label{EqMultComputations} \\
			& &+2m_q(x,y)+2m_q(F,x)+2m_q(F,y) \nonumber
			\end{eqnarray}
			We used here that for any homogeneous polynomials $G,K,M$, intersection multiplicity satisfies $m_q(G,KG+M)=m_q(G,M)$ and $m_q(G,KM)=m_q(G,K)+m_q(G,M)$. One can easily check that no two terms in (\ref{EqMultComputations}) are non-zero at the same time (that is, for some $q\in D_6\cap D_H$), so we conclude that $D_6$ and $D_H$ intersect only in multiplicity exactly $2$.

			 \label{ExDHD6IntMult2}
		\end{enumerate}
	\end{example}

\section{Proof of \cref{ThmMainPrymSubFibration}}\label{sectionConstructionf}

In this section, we freely use the notation introduced in \S \ref{SectionGeomSettings}.

\subsection{Construction of $\calC$ and $i$}\label{SectionConstructionOfRelativeCurve}

There is a divisor $F_P\subset F(X)$ defined as the closure
$$F_P\coloneqq \overline{\{[L] \in F(X) \mysetminus P^\vee \ | \ L\cap P\neq \emptyset \}},$$
where $P^\vee\subset F(X)$ is the dual plane of lines contained in $P\subset X$. There is a natural map $\widetilde{F}\to F_P$, which is bijective over $F_P\mysetminus P^\vee$ (for general $(X,P)$, it is an isomorphism). Note that for a hyperplane $H\in\calU$ and $Y=X\cap H$, the curve $C_H$ is given by the base change
$$\begin{tikzcd}
C_H \ar[d,hook] \ar[r] & F_P \ar[d,hook] \\
F(Y) \ar[r,hook] & F(X),
\end{tikzcd}$$
hence for $H\in\calU$ we can consider $C_H$ as a curve in $\widetilde{F}$. Indeed, if $L\in P^\vee\cap F_P$ then $L\in C_H$, in particular $L$ is of the second type in $Y=X\cap H$, which does not happen for $H\in \calU$ (see \S \ref{SectionGeneralityAssumptions}).

Working in a family, one can consider the fibre diagram
$$\begin{tikzcd}
\calC \ar[d,hook] \ar[r] & F_P \ar[d,hook] \\
F_\calU(\calY) \ar[r] & F(X),
\end{tikzcd}$$
so that $\calC\to \calU$ is a family of curves, whose fibre over $H\in\calU$ is $C_H$. Since each fibre lies in $\widetilde{F}$, the family $\calC$ naturally lives in $\widetilde{F}\times \calU$. The image of $\calC$ by 
$$\widetilde{F} \times \calU \xrightarrow{\pi\circ \alpha \times \Id_\calU}\matP^2\times \calU$$
is a family of curves $\calD\to \calU$ whose fibre over $H\in\calU$ is $D_H$. 
The family $\calD$ induces a map $\calU\to |5H_{\matP^2}|$.
The map $\calC \to \calD$ is a family of étale double covers, from which we can construct the relative Prym variety $\Prym_\calU(\calC/\calD)$, given fibrewise by $\Prym(C_H/D_H)$ \cite[\S 4.1]{LazaSaccaVoisinHKCompactificationIntermediateJacobianFibrationCubic4fold}. 
\newline


\begin{proposition}\label{PropUto5HP2isfinite}
	The map $\calU \to |5H_{\matP^2}|$ is quasi-finite.
\end{proposition}

\begin{proof}
	Pick two hyperplanes $H_1$ and $H_2$ and assume $D_H\coloneqq D_{H_1}=D_{H_2}$. 
	Then there exists a point $y\in D_{H}\mysetminus D_6$ such that the fibre of $\pi_X$ over $y$ is a smooth conic $Q_y$, and $Q_y\cap H_i=L_i\cup L_i'$ with $L_i\cap L_i'\neq \emptyset$, for $i=1,2$. Up to reordering, since these four lines lie in $Q_y$, we can assume that $L_1\cap L_2=\{x\}$ and $L_1'\cap L_2'=\{x'\}$. We obtain a birational map
	$$C_{H_1}\dashrightarrow C_{H_2}$$
	over $D_H \mysetminus D_6$, defined over $y$ by sending $L_1$ to $L_2$ and $L_1'$ to $L_2'$, compatible with the covering involutions on each curve. Since both curves are smooth, we obtain an isomorphism $C_{H_1}\simeq C_{H_2}$ and therefore $\Prym(C_{H_1}/D_H)\simeq \Prym(C_{H_2}/D_H)$. By the Torelli theorem for cubic threefolds, we get $X\cap H_1\simeq X\cap H_2$. Up to shrinking $\calU$, for fixed $H_1$, there exist only a finite number of such $H_2$ by maximal variation for hyperplane sections of cubic fourfolds containing a plane \cite[Prop. 5.18]{HuybrechtsCubicHypersurfaces}.
\end{proof}

The map $\calC \to \calD$ factors through $\widetilde{F} \times \calU \to S \times \calU$. The image of $\calC$ by the latter map is a family of curves in $S$ lying in the linear system $\pi^*|5H_{\matP^2}|\subset |5H_S|$, therefore it yields a map
$$i\colon \calU \to |5H_S|$$
which is quasi-finite as a consequence of \cref{PropUto5HP2isfinite}. Set $\calV\coloneqq i(\calU)\subset |5H_S|$, and we denote $\overline{C}_H\coloneqq \alpha(C_H)$.

\subsection{Description of $\calV$}

Fix a hyperplane $H$, and set as before $L\coloneqq P\cap H$. Suppose that two lines $L_1,L_2\in C_H$ have the same image $y\in\matP^2$ by $\widetilde{F}\to\matP^2$. By construction, $y\in D_H$. For $i=1,2$, $L_i\in Q_y\cap H$. If $Q_y$ is smooth (i.e. $y\notin D_6$), then $L_1$ and $L_2$ are in the two different components of $F(Q_y)$, so in particular $\alpha(L_1)\neq \alpha(L_2)$ in $S$. If $Q_y$ is not smooth, then it is a cone over a conic and $L_1$ and $L_2$ are two distinct lines of the ruling. In particular, $\alpha(L_1)=\alpha(L_2)$.

In other words, the map $C_H \to \overline{C}_H$ is injective over $D_H\mysetminus D_6$, and $2$-to-$1$ over $D_H\cap D_6$.

\begin{lemma}\label{LemmaIntersectionD6DHmult2}
	The plane curves $D_6$ and $D_H$ only meet with multiplicity exactly $2$.
\end{lemma}

\begin{proof}
	In view of \S \ref{SectionGeneralityAssumptions}, it is enough to prove that $D_6$ and $D_H$ intersect with multiplicity at least $2$ at every point of the intersection, that is they share the same tangent line at any point of the intersection.
	
	The quadric surface $Q_y=\pi^{-1}(y)$ in $\matP^5$ is given by equations $f_1(y), \dots,f_p(y)$ (we will omit $y$ in the notation, but we must think that everything depends on $y\in\matP^2$). Consider the Jacobian matrix
	$$J\coloneqq \begin{pmatrix}
	\partial_1f_1 & \cdots & \partial_6f_1 \\
	\vdots & & \vdots \\
	\partial_1f_p & \cdots & \partial_6f_p
	\end{pmatrix}.$$
	The equations for $Q_y\cap H$ in $\matP^5$ are $f_1,\dots,f_p,f_{p+1}$ with $f_{p+1}$ the linear equation of $H$. The Jacobian matrix for the equations of $Q_y \cap H$ is
	$$J_H\coloneqq \begin{pmatrix}
	\partial_1f_1 & \cdots & \partial_6f_1 \\
	\vdots & & \vdots \\
	\partial_1f_p & \cdots & \partial_6f_p \\
	\partial_1f_{p+1} & \cdots & \partial_6 f_{p+1}
	\end{pmatrix}.$$
	If $Q_y$ is singular, then any minor $M(y)$ of $J(y)$ of order $3$ vanishes. Fix a point $y\in D_6\cap D_H$ and let $\vec{v}$ be a tangent vector of $\matP^2$ at $y$, tangent to $D_6$. Since $D_6$ parametrizes singular quadrics, we must have $(\partial_{\vec{v}}M)(y)=0$.
	
	Since $y\in D_H$, we know that any minor $M_H(y)$ of $J_H(y)$ of order $4$ vanishes. Such a minor must have the form
	$$M_H = \det \begin{pmatrix}
	\partial_{j_1}f_{i_1} & \cdots & \partial_{j_4}f_{i_1} \\
	\vdots & & \vdots \\
	\partial_{j_1}f_{i_4} & \cdots & \partial_{j_4}f_{i_4}
	\end{pmatrix}$$
	for suitable indices.
	Now, we compute $\partial_{\vec v}M_H$ expanding the determinant with respect to the last line.
	\begin{eqnarray*}
		\partial_{\vec v}M_H&=&\partial_{\vec v} \sum_{k=1}^{4} \varepsilon_k\partial_{j_k}f_{i_4}M_k \\
		&=& \sum_{k=1}^{4}\varepsilon_k\Big( \partial_{\vec v} \left( \partial_{j_k}f_{i_4} \right) M_k+ \left(\partial_{j_k}f_{i_4}\right)\partial_{\vec v}M_k\Big).
	\end{eqnarray*}
	where $\varepsilon_k$ is $1$ or $(-1)$ and $M_k$, $k=1,\dots, 4$ are minors of order $3$, in which the derivatives of $f_{p+1}$ do not appear. In particular, the $M_k$'s are minors of order $3$ of $J$. Hence, these minors are all $0$ at $y$ ($Q_y$ is singular), and their derivatives along $\vec{v}$ at $y$ are $0$ ($\vec{v}$ is tangent to $D_6$). Therefore $\partial_{\vec v}M_H(y)=0$, so we obtain that $\vec{v}$ is also tangent to $D_H$. 
\end{proof}

\begin{remark}
	The space of quintics $|\calO_{\matP^2}(5)|$ is $20$-dimensional. Since the condition \textit{"having at least} $1$ \textit{point of intersection of multiplicity} $k$ \textit{with a given curve"} is a codimension $k-1$ condition, the expected dimension of the locus of quintics $15$-fold tangent to $D_6$ is $5$ (in fact, we prove in \cref{LemmaLinearSystemD515TangD6isdim5} that $5$ is the actual dimension). Moreover, any other possibility of intersection of a quintic with $D_6$, with multiplicity at least $2$ at each point, gives an expected dimension strictly smaller than $5$.
\end{remark}

As a consequence of \cref{LemmaIntersectionD6DHmult2} and \cref{LemmaMult2IntImpliesNodalOnK3}, the curve $\overline{C}_H$ is a $15$-nodal curve on $S$, each node lying over the branched locus $D_H\cap D_6$.

\begin{proposition}\label{PropVIsOpenInP5VeroLinEmbed}
	The image $\calV\subset |5H_S|$ is an open subset in some $\matP^5$, which injects in $|5H_S|\simeq \matP^{26}$ by the composition of a Veronese embedding $\matP^5 \hookrightarrow \matP^{20}$ and a linear injection $\matP^{20}\hookrightarrow \matP^{26}$.
\end{proposition}

\begin{proof}
	
The image $\calV$ of $\calU$ in $|5H_S|$ is exactly the pullback by $\pi\colon S\to\matP^2$ of the image of $\widetilde{i}\colon\calU \to |5H_{\matP^2}|$. Fix a point $H\in\calU$ and consider the intersection $D_6\cap D_H$, seen as a divisor on $D_6$. There is an effective reduced divisor $Z$ on $D_6$ of degree $15$ satisfying $2Z=5h$, for $h\coloneqq H_{\matP^2}|_{D_6}$ the hyperplane class on $D_6$.
	
	\begin{lemma}\label{LemmaLinearSystemD515TangD6isdim5}
		The linear system $|Z|$ has dimension $5$.
	\end{lemma}
	
	\begin{proof}
		By Riemann-Roch, we know that $(h^0-h^1)(D_6,\calO_D(Z))=6$, so we must prove that $H^1(D_6,\calO_{D_6}(Z))=0$. If not, by Serre duality we obtain an effective divisor $R\sim K-Z$ of degree $3$, that is $R=p+q+r$ for $p,q,r\in D_6$ three points. Note that $2K-2Z \sim 6h - 5h \sim h$. Since $H^0(\matP^2,\calO_{\matP^2}(1)) \to H^0(D_6,\calO_{D_6}(1))$ is surjective, there must exist a line $l\subset \matP^2$ such that $l\cap D_6=2(p+q+r)$ as divisors on $D_6$, in other words the line $l$ must be tritangent to $D_6$. We avoid this situation by genericity of $D_6$ (see \S \ref{SectionGeneralityAssumptions}).
	\end{proof}
	
	Recall that $H^0(\matP^2,\calO_{\matP^2}(5)) \to H^0(D_6,\calO_{D_6}(5))$ is an isomorphism, and denote $\operatorname{Ver}\colon |Z| \hookrightarrow H^0(\matP^2,\calO_{\matP^2}(5))$ the multiplication by $2$ map (it is a degree $2$ embedding, because two quintics cannot be tangent in $15$ points). Then the image $\widetilde{i}(C_H)$ lies in $\operatorname{|Z|}$.
	Starting with another hyperplane $H'$ gives another divisor $Z'$ on $D_6$. But $2Z\sim 2Z'\sim 5h$, and a priori the divisor classes of $Z$ and $Z'$ differ by a $2$-torsion divisor on $D_6$ (in particular, $\operatorname{Ver}(|Z|)=\operatorname{Ver}(|Z'|)$ if and only if $Z=Z'$). However, $\Pic^0(D_6)[2]\simeq \matZ/2\matZ^{\oplus 20}$ is finite, hence $\widetilde{i}(\calU)\subset |Z|$ by connectedness.
	We conclude that $i\colon \calU \to |5H_S|$ is the composition of $\calU \to |Z|$, $\operatorname{Ver}\colon |Z| \hookrightarrow |5H_{\matP^2}|$ (Veronese embedding) and $\pi^*\colon |5H_{\matP^2}| \hookrightarrow |5H_S|$ (linear embedding). Note that the image of $\calU \to |Z|$ is open by \cite[\href{https://stacks.math.columbia.edu/tag/0F32}{Tag 0F32}]{stacks-project}, because it is quasi-finite by \cref{PropUto5HP2isfinite} and dominant (both spaces have the same dimension).
\end{proof}

\begin{corollary}
	The locally closed 
	subspace $\calV\subset |5H_S|$ is smooth.	
\end{corollary}

Denote $\overline{\calC}$ the restriction of the universal curve over $|5H_S|$ to $\calV$. Consider the fibrewise resolution $\nu_\calU\colon \overline{\calC}^\nu \to \overline{\calC}$ over $\calU$ (see \cref{LemmaFibrewiseNormalizationCurves}). We denote $\overline{C}_H$, resp. $\overline{C}_H^\nu$, the fibre $\overline{\calC}_{i(H)}$, resp. $\overline{\calC}_{i(H)}^\nu$. 

%
%
\begin{remark}\label{RmkFamCisFibreProdCnu}
	The proof of \cref{PropUto5HP2isfinite} in fact shows that $i(H_1)=i(H_2)$ implies $C_{H_1}\simeq C_{H_2}$, therefore $\calC_{H}= C_H \simeq\overline{C}_{H}^\nu$. In particular, we have $\calC\simeq \overline{\calC}^\nu\times_\calV\calU$.
\end{remark}


Denote $M\coloneqq M_S(0,|5H_S|,-10)$ the moduli space of stable sheaves with respect to a generic fixed polarization (see \S \ref{SectionCompactifiedJacobianLagrangianFibrations}) . 

\begin{proof}[Proof of \cref{ThmMainPrymSubFibration}]\label{ProofMainThm1}
	First, by \cref{RmkFamCisFibreProdCnu} we get a natural map $\Pic_\calU^0(\calC) \to \Pic^0_\calV(\overline{\calC}^\nu)$, which is quasi-finite by \cref{PropUto5HP2isfinite}. As a direct application of \cref{ThmSubfibrationJacobianNormalizations}, we get an embedding $\Pic^0_\calV(\overline{\calC}^\nu)\hookrightarrow M$. Moreover, by \cref{ThmPrymSubfibrationOverNodalCurves}, the image in $M$ of $\Prym_\calU(\calC/\calD)\subset \Pic^0_\calU(\calC)$ by $f$ is a symplectic (hence smooth) subvariety.
\end{proof}

\section{Questions and related works}\label{SectionQuestions}

\subsection{About \cref{ThmPrymSubfibrationOverNodalCurves}}\label{SectionQuestionGenPrymFibConstruction}

We are confident that the construction of \S \ref{SectionSubfibPrymVar} can be performed in greater generality. Namely, we could assume that $S\to X$ is a double cover of a del Pezzo surface, or even an Enriques surface (in the latter case, the curves in $X$ must also have nodes, and one should consider Prym varieties of nodal curves), up to some assumptions on the curves. It is not too hard to prove, in nice cases, that the base $\calV$ is an open subset of a projective space. This gives good hopes for the following.

\begin{question}\label{QuestionSmoothCompactPrymFibration}
	Does the variety $P$ of \cref{ThmPrymSubfibrationOverNodalCurves} admit a smooth hyperk\"ahler compactification?
\end{question}

The results in the present article must be compared to \cite{ASFRelativePrymVar}. In \textit{loc. cit.}, the authors construct a Lagrangian fibration $P_E$ in Prym varieties as a subfibration of a Beauville--Mukai system on a K3 surface lying over an Enriques surface. Moreover, they study the singularities of the closure $\overline{P}_E$ and they prove that it either admits no symplectic desingularization (whenever the linear system is not hyperelliptic), or it admits in the hyperelliptic case a symplectic resolution, which is a hyperk\"ahler manifold of K$3^{[n]}$-type. It is interesting to see that the example we exhibit (\cref{ThmPrymSubfibrationOverNodalCurves} for $k=5$) is birational (up to a cover) to a hyperk\"ahler manifold of OG10-type (\cref{ThmMainPrymSubFibration}). These remarks yield the following question, as a first step towards \cref{QuestionSmoothCompactPrymFibration}.

\begin{question}
	Does the closure $\overline{P}$ admit the structure of an \textit{irreducible symplectic variety}?
\end{question}

We deliberately do not specify what we mean by irreducible symplectic variety, but one could think of it as the singular analogue of a hyperk\"ahler manifold.

\subsection{About \cref{ThmMainPrymSubFibration}}\label{SectionQuestionMainThm}

We have described the intermediate Jacobian fibration $\calJ\to \calU$ of a general cubic fourfold containing a plane as a subfibration of a Beauville--Mukai system only \textit{rationally} and \textit{up to a finite cover}. A first natural question is whether these hypotheses could be removed.

\begin{question}
	Is the map $\calU \to |5H_{\matP^2}|$ of \cref{PropUto5HP2isfinite} injective? Does it extend to a map $(\matP^5)^\vee \to |5H_{\matP^2}|$?
\end{question}

If these hold true, then the hyperk\"ahler compactification $\overline{\calJ}$ of $\Prym_\calU(\calC/\calD)$ constructed in \cite{SaccaBiratGeomIntJacFibCubic4f} could be thought as a symplectic desingularization of the closure of the variety $P$ constructed in \cref{ThmPrymSubfibrationOverNodalCurves}.

Let us mention here that the intermediate Jacobian fibration $\calJ\to\calU$ of a cubic fourfold $X$ containing a plane $P$ is isomorphic to its twisted version $\calJ^T\to \calU$ (parametrizing degree $1$ $1$-cycle in the hyperplane sections of $X$) studied in \cite{VoisinInterJacFibTwistCase}, since $X$ admits a relative cycle over $\calU$ given fibrewise over $H\in \calU$ by $P\cap H$. In particular, $\calJ^T\to \calU$ admits a section. Moreover, $\calJ^T\to\calU$ admits a compactification as a moduli space of stable objects on the Kuznetsov component of $\D^b(X)$ \cite{LPZEllQuinticsCubicFourfoldOG10LagFib}. The latter is identified with $\D^b(S,\alpha)$, where $S$ is the K3 surface associated to $X$ and $\alpha\in\Br(S)$ is the class associated to the Brauer-Severi variety $\widetilde{F}\to S$ (see \S \ref{SectionGeomSettings}). We wonder if one could replace the moduli space $M$ of \cref{ThmMainPrymSubFibration} with a moduli space of $\alpha$-\textit{twisted} sheaves, since the Brauer class $\alpha$ does not play any role in the present paper.

Finally, one could wonder what happens if we perform the same construction, replacing $\calJ$ with $\calJ^T$, or more generally with the space $\calJ^k$, $k\in\matZ_{\geq 0}$, which parametrizes relative degree $k$ $1$-cycles in the hyperplane sections of $X$. Fix some hyperplane section $Y=X\cap H$ and set $L=P\cap H$. Moving from $J(Y)\simeq \CH_1(Y)_{0}$ to $J^k(Y)\coloneqq \CH_1(Y)_k$ by adding $d[L]$ is equivalent to moving from $\Prym(C_H/D_H)$ to $\Prym^{5k}(C_H/D_H)$, where the latter is one component of the preimage of $\calO_{D_H}(k)$ by the norm map $C_H\to D_H$ (a torsor over $\Prym(C_H/D_H)$). Pushing these translated Pryms to the K3 surface $S$, we obtain torsion sheaves with Mukai vector $(0,5H_S,5k-10)$, and one can perform the same construction as in the paper. Note that this Mukai vector will always be non-primitive, for any $k$.




\bibliographystyle{abbrv}
\bibliography{bib_IJcubics}

\end{document}